\documentclass[11pt,a4paper]{article}

\usepackage{amsfonts}
\usepackage{amsthm}
\usepackage{amsmath}
\usepackage{amssymb}
\usepackage{amscd}
\usepackage{mathrsfs}
\usepackage[mathscr]{eucal}
\usepackage{graphicx}
\usepackage{epstopdf} 
\usepackage{pict2e}
\usepackage{epic}
\usepackage{xcolor}
\usepackage{float}

\usepackage{cite}
\usepackage{hyperref}
\hypersetup{colorlinks=true, urlcolor= black, linkcolor=black, citecolor=black}

\usepackage[utf8]{inputenc}
\usepackage[T1]{fontenc}
\usepackage{lmodern}
\usepackage{dsfont}
\usepackage{indentfirst}

\numberwithin{equation}{section}

\theoremstyle{plain}
\newtheorem{Thm}{Theorem}[section]
\newtheorem{Lem}[Thm]{Lemma}
\newtheorem{Cor}[Thm]{Corollary}
\newtheorem{Prop}[Thm]{Proposition}

\theoremstyle{definition}
\newtheorem{Def}[Thm]{Definition}

\newtheorem{Rem}[Thm]{Remark}

\newtheorem{Claim}[Thm]{Claim}

\usepackage[margin=3.1cm]{geometry}

\newcommand\bfC{\mathbf{C}}

\title{An alternative approach for the mean-field behaviour of weakly self-avoiding walks in dimensions $d>4$}
\begin{document}

\author{Hugo Duminil-Copin\footnote{Institut des Hautes \'Etudes Scientifiques, \url{duminil@ihes.fr}}\ \footnotemark[2]\footnote{Université de Genève, \url{hugo.duminil@unige.ch}, \url{romain.panis@unige.ch}} , Romain Panis\footnotemark[2]}
\maketitle

{\em We dedicate this article to Geoffrey Grimmett on the occasion of his seventieth birthday.}

\begin{abstract}
This article proposes a new way of deriving mean-field exponents for the weakly self-avoiding walk model in dimensions $d>4$. Among other results, we obtain up-to-constant estimates for the full-space and half-space two-point functions in the critical and near-critical regimes. A companion paper proposes a similar analysis for spread-out Bernoulli percolation in dimensions $d>6$ \cite{DumPan24Perco}.
\end{abstract}

\section{Introduction}
One of the main challenges of statistical mechanics consists in understanding the (near-)critical behaviour of diverse lattice models. Among other things, one may for instance compute the models' \emph{critical exponents}. Conducting this task is in general extremely difficult as it involves in a subtle way both the special features of the models and the geometry of the graphs on which they are defined.

 A striking observation was made in the case of models defined on the hypercubic lattice $\mathbb Z^d$: above a so-called \emph{upper-critical dimension} $d_c$, the geometry ceases to play a role and the critical exponents take an easier form, equalling those obtained on a Cayley tree (or \emph{Bethe lattice}) or the complete graph. The regime $d>d_c$ is called the \emph{mean-field} regime of a model. Noteworthy methods such as the \emph{lace expansion} \cite{BrydgesSpencerSAW} 
 or the \emph{rigorous renormalisation group method} \cite{BauerschmidtBrydgesSlade2014Phi4fourdim,BauerschmidtBrydgesSlade2015WSAW4D,BauerschmidtBrydgesSlade2015WSAW4DLogCorrections,BauerschmidtBrydgesSladeBOOKRG2019}
 have emerged to carry out the analysis of the mean-field regime. However, a limitation
of these methods lies in their predominantly \emph{perturbative} nature, which is reflected in the necessity of exhibiting a small parameter in the model. The Weakly Self-Avoiding Walk (WSAW) model includes such a small parameter in its definition, making it a natural testing ground for developing the analysis of the mean-field regime.

Lace expansion was successfully applied to derive the mean-field behaviour of the WSAW model in dimensions $d>4$: in \cite{BrydgesSpencerSAW}, Gaussian limit laws were exhibited for the displacement of the $n$-steps WSAW, while in \cite{HaraDecayOfCorrelationsInVariousModels2008,BolthausenvanderHofstadKozmaDummies2018,BrydgesHelmuthHolmesContinuousLaceExpansion,Slade2022NewProofCVLACEExpansion,Slade2023near}, estimates of the two-point function were obtained.

The WSAW model is also a toy model to study (\emph{strictly}) self-avoiding walks. Using the lace expansion, Slade \cite{Slade1987Diffusion} extended the results of \cite{BrydgesSpencerSAW} to the setup of the self-avoiding walk model in sufficiently large dimensions. This restriction was later removed by Hara and Slade \cite{HaraSlade1992SAW} who optimised the techniques to obtain mean-field behaviour of the SAW model in dimensions $d>4$. Let us mention that the lace expansion was also applied to a variety of models of statistical mechanics including Bernoulli percolation \cite{HaraSlade1990Perco,HaraSladevdHofstad2003PercoSO,HaraDecayOfCorrelationsInVariousModels2008,FitznervdHofstad2017Perco-d>10}, lattice trees and animals \cite{HaraSlade1990LatticeTrees}, the Ising model \cite{Sakai2007LaceExpIsing,Sakai2022correctboundsIsing}, and the $\varphi^4$ model \cite{Sakai2015Phi4,BrydgesHelmuthHolmesContinuousLaceExpansion}. For more information on the lace expansion approach, we refer to the monograph \cite{SladeSaintFlourLaceExpansion2006}.

In this paper, we provide an alternative argument to obtain mean-field bounds on the two-point function of the weakly self-avoiding walk in dimensions $d>4$. This technique extends to a number of other models after suitable modifications (see e.g.~\cite{DumPan24Perco} for the example of percolation), but we choose to stick to the case of the WSAW model to present the method in its simplest context.

\paragraph{Notations.} Consider the hypercubic lattice $\mathbb Z^d$ and let $y\sim z$ denote the fact that $y$ and $z$ are neighbours in $\mathbb Z^d$.
Set  ${\bf e}_j$ to be the unit vector with $j$-th coordinate equal to 1. Write $x_j$ for the $j$-th coordinate of $x$, and denote its $\ell^\infty$ norm by $|x|:=\max\{|x_j|:1\le j\le d\}$. Set $\Lambda_n:=\{x\in \mathbb Z^d:|x|\le n\}$ and for $x\in \mathbb Z^d$, $\Lambda_n(x):=\Lambda_n+x$. Also, set $\mathbb H_n:=-n{\bf e}_1+\mathbb H$, where $\mathbb H:=\mathbb Z_+\times\mathbb Z^{d-1}=\{0,1,\ldots\}\times \mathbb Z^{d-1}$. Finally, let $\partial S$ be the boundary of the set $S$, given by the vertices in $S$ with at least one neighbour outside $S$.

\subsection{Definitions and statements of the results}
 Let $\lambda\in [0,1]$. In what follows, we drop it from the notation.  Let $\mathcal{W}$ be the set of finite paths in $\mathbb Z^d$. Let $|\gamma|$ be the number of edges of $\gamma\in\mathcal{W}$. For $\gamma=(\gamma(0),\ldots,\gamma(|\gamma|))\in \mathcal{W}$, introduce the weight\begin{equation}
    \rho(\gamma):=\prod_{0\leq s<t\leq |\gamma|}\big(1-\lambda\mathds{1}_{\gamma(s)=\gamma(t)}\big).
\end{equation}
For a set $\Lambda\subset\mathbb Z^d$, $\beta\geq0$, and $x,y\in\Lambda$, define the {\em two-point function} (in $\Lambda$) by
\begin{equation}
    G_\beta^\Lambda(x,y):=\sum_{\gamma: x\rightarrow y\subset\Lambda}\beta^{|\gamma|}\rho(\gamma),
\end{equation}
where $x\rightarrow y$ means that $\gamma$ starts at $x$ and ends at $y$ (in particular, if $x=y$, we count the walk of length zero). When $\Lambda=\mathbb Z^d$, we drop it from the notation and simply write $G_\beta(x,y)$.

Let $\beta_c$ be the critical value for the finiteness of the {\em susceptibility} 
\begin{equation}\label{eq: def susceptibility}
\chi(\beta):=\sum_{x\in \mathbb Z^d}G_\beta(0,x),
\end{equation} defined by the formula
\begin{equation}
    \beta_c=\beta_c(\lambda):=\sup\Big\lbrace \beta\geq 0: \: \chi(\beta)<\infty \Big\rbrace.
\end{equation}
It is easy to obtain (see \cite{SladeSaintFlourLaceExpansion2006,BauerschmidtDCGoodmanSlade}) that $\beta_c\in[(2d)^{-1},\mu_c(d)^{-1}]$, where $\mu_c(d)$ is the connective constant of $\mathbb Z^d$. For $\beta<\beta_c$, the two-point function decays exponentially fast in the distance. A convenient quantity helping to monitor the rate of decay is the \emph{sharp length} $L_\beta$ defined below (see also \cite{DuminilTassionNewProofSharpness2016,PanisTriviality2023,DuminilPanis2024newLB} for a study of this quantity in the context of Bernoulli percolation and the Ising model). For $\beta\geq 0$ and $S\subset \mathbb Z^d$, let
\begin{equation}
	\varphi_\beta(S):=\sum_{\substack{y\in S\\ z\notin S\\ y\sim z}}G^S_\beta(0,y)\beta.
\end{equation}
The sharp length $L_\beta$ is defined by
\begin{equation}\label{eq: def L beta}
    L_\beta:=\inf\{k\ge1:\varphi_\beta(\Lambda_k)\le1/e^2\}.
\end{equation}
Exponential decay of the two-point function guarantees that  $L_\beta$ is finite for $\beta<\beta_c$, and it is easy to prove that it is infinite\footnote{The existence of $k$ such that $\varphi_{\beta_c}(\Lambda_k)<1$ would imply that $\chi(\beta_c)\leq \tfrac{|\Lambda_k|}{1-\varphi_{\beta_c}(\Lambda_k)}$ by Lemma \ref{Lem: Simon-Lieb WSAW} and the strategy of \cite{DuminilTassionNewProofSharpness2016}. In particular, this yields $\chi(\beta_c)<\infty$ which is impossible, see e.g \cite[(4.6)]{BauerschmidtDCGoodmanSlade}.} for $\beta=\beta_c$ (see \cite{SimonInequalityIsing1980,DuminilTassionNewProofSharpness2016}).

We now state our first main result, which provides uniform upper bounds on the full-space and half-space two-point functions.

\begin{Thm}[Upper bounds]\label{thm:mainwsaw} Let $d>4$.
There exist $C,\lambda_0>0$ such that for every $\lambda<\lambda_0$, every $\beta\le\beta_c$, and every $x\in \mathbb Z^d$,
\begin{align}
\label{eq:bound full plane}
G_\beta(0,x)&\le \frac{C}{(1\vee |x|)^{d-2}}\exp(-|x|/L_\beta),\\
\label{eq:bound half plane} G_\beta^\mathbb H(0,x)&\le \frac{C}{(1\vee|x_1|)^{d-1}}\exp(-|x_1|/L_\beta).
\end{align}
\end{Thm}
A near-critical upper bound was derived using the lace expansion in \cite{Slade2023near} (see also \cite{Liu2023Plateau} for the case of the strictly (i.e. $\lambda=1$) self-avoiding walk model). There, $L_\beta$ is replaced by $C_0(\beta_c-\beta)^{-1/2}$ for some $C_0>0$. In fact, we will prove that the two quantities are within a multiplicative constant of each other, see Corollary \ref{cor: critical exponents}.
The second main theorem of this article provides lower bounds matching the bounds in Theorem~\ref{thm:mainwsaw}.
\begin{Thm}[Lower bounds]\label{thm:main2wsaw} Let $d>4$.
There exist $c,C,\lambda_0>0$  such that the following holds. For every $\lambda<\lambda_0$, every $\tfrac{1}{2d}\le \beta\leq\beta_c$, and every $x\in \mathbb Z^d$,
\begin{align}
\label{eq:bound full plane 2}
G_\beta(0,x)&\ge \frac{c}{(1\vee|x|)^{d-2}}\exp(-C|x|/L_\beta),\\
\label{eq:bound half plane 2}G_\beta^\mathbb H(0,x)&\ge \frac{c}{(1\vee|x|)^{d-1}}\exp(-C|x|/L_\beta)\qquad\text{ provided that }x_1=|x|.
\end{align}
\end{Thm}
The bounds of Theorems \ref{thm:mainwsaw} and \ref{thm:main2wsaw} are expected to hold also at the upper-critical dimension $d=4$ (with potential logarithmic corrections in the exponential). In the case of the critical full-space estimate, this has been successfully derived in \cite{BauerschmidtBrydgesSlade2015WSAW4D}.

A direct consequence of Theorem \ref{thm:mainwsaw} is the finiteness at criticality of the so-called \emph{bubble diagram}, which plays a central role in the study of the mean-field regime of the WSAW model, see e.g.~\cite{BovierFelderFrohlich1984BubbleFiniteWSAW,SladeSaintFlourLaceExpansion2006,BauerschmidtDCGoodmanSlade}.
\begin{Cor}[Finiteness of the bubble diagram]\label{cor: finiteness bubble} Let $d>4$. There exists $\lambda_0>0$ such that for every $\lambda<\lambda_0$,
\begin{equation}
	B(\beta_c):=\sum_{x\in \mathbb Z^d}G_{\beta_c}(0,x)^2<\infty.
\end{equation}
\end{Cor}

We now describe how to recover the mean-field behaviour of the {\em susceptibility} (defined in \eqref{eq: def susceptibility}) and the {\em correlation length} $\xi_\beta$ defined for $\beta<\beta_c$ by\footnote{The limit is shown to exist by a classical subadditivity argument, see \cite[Chapter~4]{MadrasSlade2013SAW}.} 
\begin{equation}\label{eq: def corr length wsaw}
	\xi_\beta^{-1}:=\lim_{n\rightarrow \infty}-\tfrac{1}{n}\log G_\beta(0,n\mathbf{e}_1).
\end{equation}
\begin{Cor}\label{cor: critical exponents} Let $d>4$. 
There exist  $c,C,\lambda_0>0$ such that the following holds. For every $\lambda<\lambda_0$, and every $\frac{1}{2d}\le \beta<\beta_c$,
\begin{align}\label{eq: bounds susceptibility}
	c(\beta_c-\beta)^{-1}&\leq \chi(\beta)\leq C(\beta_c-\beta)^{-1},
	\\ \label{eq: bounds xi}
	c(\beta_c-\beta)^{-1/2}&\leq \xi_\beta\leq C(\beta_c-\beta)^{-1/2},
	\\ \label{eq: bounds L}
	c(\beta_c-\beta)^{-1/2}&\leq L_\beta\leq C(\beta_c-\beta)^{-1/2}.
\end{align}
\end{Cor}
\begin{proof} Let $d>4$ and $\lambda_0$ be given by Corollary \ref{cor: finiteness bubble}. It is classical (see e.g.~\cite[Sections~4.1--4.2]{BauerschmidtDCGoodmanSlade}) that for $\beta<\beta_c$,
\begin{equation}\label{eq: diff on susc}
\frac{\chi(\beta)^2}{B(\beta)}\le \frac{{\rm d}\chi}{{\rm d}\beta}\le  \chi(\beta)^2.
\end{equation}
Combined with Corollary \ref{cor: finiteness bubble}, which gives that $B(\beta)\leq B(\beta_c)<\infty$, \eqref{eq: diff on susc} readily implies \eqref{eq: bounds susceptibility}. 
%

The bounds \eqref{eq: bounds xi} and \eqref{eq: bounds L} are obtained using \eqref{eq: bounds susceptibility}, and Theorems \ref{thm:mainwsaw} and \ref{thm:main2wsaw} twice. More precisely, using Theorems \ref{thm:mainwsaw} and \ref{thm:main2wsaw} together with \eqref{eq: def corr length wsaw} we obtain that 
\begin{equation}\label{eq: two correl length are asymp wsaw}
	\xi_\beta\asymp L_\beta
\end{equation}
for $\tfrac{1}{2d}\leq\beta< \beta_c$, where $\asymp$ means that the ratio of the quantities is bounded away from $0$ and $\infty$ by two constants that are independent of $\beta$. Then, using again Theorem \ref{thm:mainwsaw}, we obtain that
\begin{equation}
	\chi(\beta)\leq\sum_{x\in \mathbb Z^d}\frac{C}{(1\vee |x|)^{d-2}}e^{-|x|/L_\beta}\leq C_1L_\beta^2,
\end{equation}
where $C_1>0$. Moreover, Theorem \ref{thm:main2wsaw} gives the existence of $c_1>0$ such that for $\tfrac{1}{2d}\leq \beta< \beta_c$,
\begin{equation}
	\chi(\beta)\geq \sum_{x\in \Lambda_{L_\beta}}\frac{ce^{-C}}{(1\vee |x|)^{d-2}}\geq c_1 L_\beta^2.
\end{equation}

Combining the two previously displayed equations gives
\begin{equation}
	\chi(\beta)\asymp L_\beta^2
\end{equation}
for $\tfrac{1}{2d}\leq\beta< \beta_c$, which also gives $\chi(\beta)\asymp \xi_\beta^2$ by \eqref{eq: two correl length are asymp wsaw}. The proofs of \eqref{eq: bounds xi} and \eqref{eq: bounds L} follow readily from these observations and from \eqref{eq: bounds susceptibility}. 
\end{proof}

Let us mention that improved versions of \eqref{eq: bounds susceptibility} and \eqref{eq: bounds xi} have been derived by Hara and Slade \cite{HaraSlade1992SAW} (using the lace expansion) in the context of the strictly self-avoiding walk model in dimensions $d>4$. Indeed, \cite[Theorem~1.2]{HaraSlade1992SAW} provides exact asymptotic formulae for these quantities as $\beta$ approaches $\beta_c$.

\subsection{The fundamental inequalities}
A crucial role will be played by the following two inequalities, see Figure \ref{fig: simon-lieb} for an illustration. For completeness, we include the  proof of this (classical) statement in Appendix \ref{appendix:sl}.

\begin{Lem}\label{Lem: Simon-Lieb WSAW} Let $d\geq 2$. For $0<\beta<\beta_c$, $0\in S\subset \Lambda$ with $\Lambda\subset \mathbb Z^d$, and $x\in \Lambda$, 
\begin{align}\label{eq:SL}
G_\beta^\Lambda(0,x)&\le G_\beta^S(0,x)+ \sum_{\substack{y\in S\\ z\in \Lambda\setminus S\\ y\sim z}}G_\beta^S(0,y)\beta G_\beta^\Lambda(z,x),\\
G_\beta^\Lambda(0,x)&\ge G_\beta^S(0,x)+\sum_{\substack{y\in S\\ z\in \Lambda\setminus S\\ y\sim z}}G_\beta^S(0,y)\beta G_\beta^\Lambda(z,x)- \lambda \sum_{u\in S}E_\beta^{S,\Lambda}(u)G_\beta^\Lambda(u,x),\label{eq:reversed SL}
\end{align}
where 
\begin{equation}
E_\beta^{S,\Lambda}(u):=\sum_{\substack{y \in S\\ z\in \Lambda\setminus S\\ y\sim z}}G_\beta^S(0,u)G_\beta^S(u,y) \beta G_\beta^\Lambda(z,u).
\end{equation}
\end{Lem}

The quantity
\begin{equation}
E_\beta^{S,\Lambda}:=\sum_{u\in S}E_\beta^{S,\Lambda}(u)
\end{equation}
will be referred to as the {\em error amplitude}. This quantity will be shown to be finite when $d>4$, which is responsible for the restriction on the dimension in this paper (see \eqref{eq: where d>4 is important}). Controlling the size of this error amplitude will be crucial to the argument.

\begin{figure}
	\begin{center}
		\includegraphics{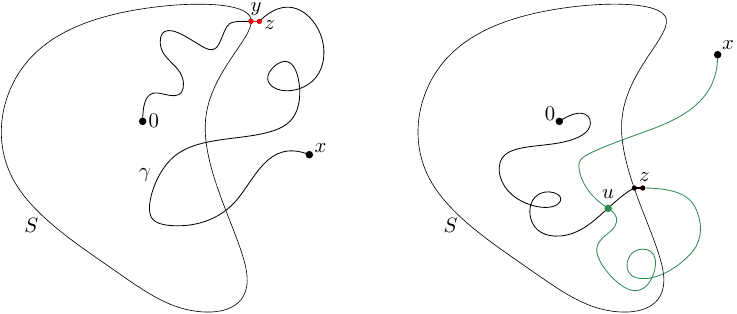}
		\label{fig: simon-lieb}
		\caption{A depiction of Lemma \ref{Lem: Simon-Lieb WSAW}. On the left we illustrate how we decompose a path $\gamma$ contributing to $G_\beta^{\Lambda}(0,x)-G^S_\beta(0,x)$. The first edge leaving $S$ is represented in red. On the right we illustrate a configuration in which the portions of $\gamma$ from $0$ to $z$ (in black) and from $z$ to $x$ (in green) interact through an intersection inside $S$. This situation contributes to the \emph{error} term in \eqref{eq:reversed SL}.}
	\end{center}
\end{figure}

\paragraph{Acknowledgements.} Early discussions with Vincent Tassion have been fundamental to the success of this project. We are tremendously thankful to him for these interactions without which the paper would never have existed. We warmly thank Gordon Slade for stimulating discussions and for useful comments. We also thank Gady Kozma, Christophe Garban, Trishen S. Gunaratnam, Ioan Manolescu, Aman Markar, Christoforos Panagiotis,  Alexis Prévost, Grega Saksida, Florian Schweiger, Vedran Sohinger, Daniel Ueltschi, and an anonymous referee for many useful comments on an earlier version of this paper. This project has received funding from the Swiss National Science Foundation, the NCCR SwissMAP, and the European Research Council (ERC) under the European Union’s Horizon 2020 research and innovation programme (grant agreement No. 757296). HDC acknowledges the support from the Simons collaboration on localization of waves.

\section{Proof of Theorem~\ref{thm:mainwsaw}}

We will use a bootstrap argument (following the original idea from \cite{Slade1987Diffusion}) and prove that an a priori estimate on the two-point function can be improved for sufficiently small $\lambda$. One original feature of our proof is that the key quantities we track in the bootstrap involve the half-space rather than the full-space. This contrasts for instance with the lace expansion, in which all arguments require the full-space two-point function.

The idea will be to observe that the two inequalities of Lemma \ref{Lem: Simon-Lieb WSAW} provide a good control on $\varphi_\beta(\mathbb H_n)$, which can be interpreted as an averaged (or $\ell^1$) estimate on the half-space two-point function at distance $n$. The point-wise (or $\ell^\infty$) half-space estimate \eqref{eq:bound half plane} will follow from a \emph{regularity} property which allows to compare two-point functions ending at close points. Finally, we will deduce the full-space estimate from the half-space one.

To implement this scheme, we introduce the following quantity $\beta^*$.

\begin{Def} Fix $d>4$ and $\lambda\in(0,1)$. Let $C>0$ to be fixed later\footnote{One may think that $C$ will be chosen first to be very large, and then $\lambda$ small enough.}. For every $\beta\geq 0$, we introduce the conditions \eqref{eq:H_beta} and \eqref{eq:H_beta-''} defined as follows:
\begin{align}
\varphi_\beta(\mathbb H_n)&< 1+\frac{1}{2d} &\forall n\ge0,\label{eq:H_beta}\tag{$\ell^1_\beta$}\\
 G_\beta^{\mathbb H_n}(0,x)&< \frac{C}{(1\vee n)^{d-1}} &\forall n\ge0,\forall x\in\partial\mathbb H_n\label{eq:H_beta-''}\tag{$\ell^\infty_\beta$}.
\end{align}
We let 
\begin{equation}
\beta^*=\beta^*(C,\lambda):=\sup\{\beta \in [0,\beta_c]: \textup{\eqref{eq:H_beta} and \eqref{eq:H_beta-''} hold}\}.\end{equation}

\end{Def}

\begin{Rem}
Note that $(\ell^1_{\beta})$ and $(\ell^\infty_{\beta'})$ are monotonic conditions in the sense that if they both hold, then $(\ell^1_{\beta'})$ and $(\ell^\infty_{\beta'})$ hold for every $\beta'\leq \beta$. The $\tfrac{1}{2d}$ is quite arbitrary in \eqref{eq:H_beta}. In fact, we could take any number in $(0,\tfrac{1}{2d}]$. 
\end{Rem}

 A priori, it could be that $\beta^*(C,\lambda)=0$. However, this is not the case if we choose $C$ large enough.
Indeed, we claim that $\beta^*\ge \tfrac1{2d}$ when $C$ exceeds a large enough constant $\bfC_{\rm RW}=\bfC_{\rm RW}(d)>0$, as a bound by the corresponding random walk quantities implies that the estimates are true at $\beta=\frac1{2d}$.

 More precisely, notice that $\varphi_\beta(\mathbb H_n)$ and $G_\beta^{\mathbb H_n}(0,x)$ are maximal when $\lambda=0$, which corresponds to the simple random walk. Let us denote by $\mathbb P_0$ the law of the $d$-dimensional simple random walk $(X_k)_{k\geq 0}$ started at $0$. Let $n\geq 0$ and define $\tau^n:=\inf\{k\geq 1: X_k\notin \mathbb H_n\}$. For $\beta=\tfrac{1}{2d}$,
 \begin{equation}\label{eq:phi beta rw =1}
 	\varphi_{\beta,\lambda=0}(\mathbb H_n)=\mathbb P_0[\tau^n<\infty]=1,
 \end{equation}
 where the first equality follows from Markov's property, and the second one from a Gambler's ruin estimate (see \cite[Theorem~5.1.7]{LawlerLimicRandomWalks2010}) applied to the random walk induced by the projection of $X$ on its first coordinate. Moreover, Proposition \ref{prop: halfspace green function estimate} implies the existence of $\bfC_{\rm RW}=\bfC_{\rm RW}(d)>0$ such that for all $n\geq 1$ and for all $x \in \partial \mathbb H_n$, one has 
 \begin{equation}
 	G_{\beta,\lambda=0}^{\mathbb H_n}(0,x)=\mathbb E_0\Big[\sum_{\ell<\tau^n}\mathds{1}_{X_\ell=x}\Big]<\frac{\bfC_{\rm RW}}{(1\vee n)^{d-1}}.
 \end{equation}
Hence, choosing $C>\bfC_{\rm RW}$ implies that for all $\lambda\in(0,1)$, one has $\beta^*(C,\lambda)\geq \tfrac{1}{2d}$.

Below (see Lemma \ref{lem: full plane 2pt function out of halfplane one}), we will show how to derive an upper bound on the full-space two-point function under the hypothesis that $\beta\leq \beta^*(C,\lambda)$. Thus, in order to prove Theorem \ref{thm:mainwsaw}, it suffices to show that $\beta^*$ is equal to $\beta_c$ provided that $C$ and $\lambda$ are properly chosen. 

We proceed in three steps. First, we show that the bound on $\varphi_\beta(\mathbb H_n)$ can be improved when $\beta<\beta^*$.  Second, we control the gradient of the two-point function. Third, we use the fact that the two-point function does not fluctuate too much (thanks to the second point) to obtain an improved $\ell^\infty$ estimate. From these improvements, we obtain that $\beta^*$ cannot be strictly smaller than $\beta_c$, since in this case the improved estimates would remain true for $\beta$ slightly larger than $\beta^*$, which would contradict the definition of $\beta^*$.

\subsection{Improving the bound on $\varphi_\beta(\mathbb H_n)$}\label{sec:improvingphiwsaw}

This section is the crucial step of our strategy: from the bounds \eqref{eq:H_beta} and \eqref{eq:H_beta-''}, we obtain a bound on $\varphi_\beta(\mathbb H_n)$ that involves the parameter $\lambda$. For small $\lambda$, this bound is an improvement on \eqref{eq:H_beta}. Recall that $\bfC_{\rm RW}$ satisfies the following property: for every $C>\bfC_{\rm RW}$ and every $\lambda\in (0,1)$, $\beta^*(C,\lambda)\geq \tfrac{1}{2d}$.

\begin{Prop}[Improving the bound on $\varphi_\beta(\mathbb H_n)$]\label{prop:bound phi}
Fix $d>4$ and $C>\bfC_{\rm RW}$. There exists $K=K(C,d)<\infty$ such that for every $\lambda\in (0,1)$, every $\beta<\beta^*(C,\lambda)$, and every $n\ge0$, 
\begin{align}
\varphi_\beta(\mathbb H_n)&<1+K\lambda,\\
\sup\{\varphi_\beta(B):B\in \mathcal B\}&\le 1+K\lambda\label{eq:bound varphi box wsaw},
\end{align}
where $\mathcal B$ is the set of blocks, that is $\mathcal B:=\{(\prod_{i=1}^d [a_i,b_i])\cap \mathbb Z^d: \forall 1\leq i \leq d,\: a_i\leq 0\leq b_i \}$.
\end{Prop}

In the rest of Section \ref{sec:improvingphiwsaw}, we fix $\lambda\in (0,1)$ and drop it from the notation. We start with a number of elementary bounds on the two-point function in the bulk and in half-space induced by the assumption that $\beta<\beta^*$.
\begin{Lem}\label{lem: full plane 2pt function out of halfplane one}Fix $d>4$ and $C>0$. For every $\beta<\beta^*(C)$, 
\begin{align}\label{eq: full plane estimate from half plane}
    G_\beta(0,x)&\le \frac{3C}{(1\vee|x|)^{d-2}}\qquad&\forall x\in \mathbb Z^d,\\ \label{eq: half plane at distance $k$ from half space}
    G_\beta^{\mathbb H_n}(0,x)&\le \frac{2C(k+1)}{(n-k)^{d-1}}\qquad&\forall x\in \partial\mathbb H_{n-k}\text{ with }1\le k<n,\\ \label{eq: sum at distance k from halfplane}
    \sum_{x\in \partial\mathbb H_{n-k}}G_\beta^{\mathbb H_n}(0,x)\beta&\le 4(k+1)\qquad&\text{ with }n,k\ge0.
\end{align}
\end{Lem}

\begin{proof} Let us start with \eqref{eq: full plane estimate from half plane}. Without loss of generality, assume that $x_1=|x|$. Consider a path from $0$ to $x$.
If the path is not included in $\mathbb H$, decompose it according to the first edge $yz$ where $z$ is a left-most point; see Figure \ref{fig: Lemma1}. We first consider the case $|x|\geq 1$. This yields
\begin{align}    G_\beta(0,x)&\stackrel{\phantom{\eqref{eq:H_beta-''}}}\le G^{\mathbb H}_\beta(0,x)+\sum_{n\ge 0}\sum_{\substack{y\in \mathbb H_{n}\\ z\notin \mathbb H_n\\ y\sim z}} G_\beta^{\mathbb H_{n}}(0,y)\beta  G_\beta^{\mathbb H_{n+1}}(z,x)\\
    &\stackrel{\eqref{eq:H_beta-''}}\le\frac{C}{(1\vee|x|)^{d-1}}+ \sum_{n\ge 0}\varphi_\beta(\mathbb H_n)\frac{C}{(|x|+n+1)^{d-1}}\\
    &\stackrel{\eqref{eq:H_beta}}\le \frac{C(1+\tfrac{1+(2d)^{-1}}{d-2})}{|x|^{d-2}}\le \frac{2C}{|x|^{d-2}},\label{eq:key inequality half to full}
    \end{align}
  where in the penultimate inequality, we used that $\sum_{n\geq \alpha}\tfrac{1}{(n+1)^{d-1}}\leq \tfrac{1}{(d-2)\alpha^{d-2}}$ for every $\alpha\ge1$. When $x=0$, we may run the same argument except that now we use the bounds $\sum_{n\geq 0}\tfrac{1}{(n+1)^{d-1}}\leq 1+\tfrac{1}{d-2}$ and $(1+\tfrac{1}{2d})(1+\tfrac{1}{d-2})\leq 2$ (when $d>4$) to obtain that
  \begin{equation}
  	\sum_{n\geq 0}\varphi_\beta(\mathbb H_n)\frac{C}{(n+1)^{d-1}}\leq 2C.
  \end{equation}
  This implies, using \eqref{eq:key inequality half to full}, that $G_\beta(0,0)\leq 3C$, and concludes the proof of \eqref{eq: full plane estimate from half plane}.
    
  We turn to \eqref{eq: half plane at distance $k$ from half space}. Let $1\leq k <n$. Pick $x\in \partial\mathbb H_{n-k}$. To bound $G_\beta^{\mathbb H_n}(x,0)=G_\beta^{\mathbb H_n}(0,x)$, decompose the path from $x$ to 0 in the same fashion as above (see Figure \ref{fig: Lemma1}) to get
    \begin{align}\label{eq: decomposition at distance k from hyperplane}
  	G^{\mathbb H_n}_\beta(x,0)&\stackrel{\phantom{\eqref{eq:H_beta-''}}}\leq G^{\mathbb H_{n-k}}_\beta(x,0)+\sum_{j=1}^{k}\sum_{\substack{y\in \mathbb H_{n-j}\\z\notin \mathbb H_{n-j}\\ y\sim z}}G_\beta^{\mathbb H_{n-j}}(x,y)\beta G_\beta^{\mathbb H_{n-j+1}}(z,0)
  	\\&\stackrel{\eqref{eq:H_beta-''}}\leq\frac{C}{(n-k)^{d-1}}+\sum_{j=1}^{k} \Big(\sum_{\substack{y\in \mathbb H_{n-j}\\z\notin \mathbb H_{n-j}\\ y\sim z}}G_\beta^{\mathbb H_{n-j}}(x,y)\beta\Big) \frac{C}{(n-j+1)^{d-1}}
  	\\&\stackrel{\eqref{eq:H_beta}}\leq \frac{C(1+(2d)^{-1})(k+1)}{(n-k)^{d-1}}\le \frac{2C(k+1)}{(n-k)^{d-1}}.
  \end{align}
  
  For the proof of \eqref{eq: sum at distance k from halfplane}, consider the same decomposition \eqref{eq: decomposition at distance k from hyperplane} and sum it over $x\in \partial \mathbb H_{n-k}$, then use \eqref{eq:H_beta} twice instead of \eqref{eq:H_beta} and \eqref{eq:H_beta-''}. When\footnote{We used that \eqref{eq: decomposition at distance k from hyperplane} is also valid when $k=n$.} $1\leq k \leq n$, this gives 
 \begin{align}
 	\sum_{x\in \partial \mathbb H_{n-k}}G_\beta^{\mathbb H_n}(0,x)&\stackrel{\eqref{eq: decomposition at distance k from hyperplane}}\leq \frac{\varphi_\beta(\mathbb H_{n-k})}{\beta}+ \sum_{j=1}^k\varphi_\beta(\mathbb H_{k-j})\frac{\varphi_\beta(\mathbb H_{n-j+1})}{\beta}
 	\\&\stackrel{\eqref{eq:H_beta}}\leq \frac{1}{\beta}\Big((1+(2d)^{-1})+k(1+(2d)^{-1})^2\Big)\leq \frac{4(k+1)}{\beta}.
 \end{align}
When $k>n$ we adapt \eqref{eq: decomposition at distance k from hyperplane} and get
\begin{equation}\label{eq: second decomp third equation}
	G_\beta^{\mathbb H_n}(x,0)\leq G_\beta^{\mathbb H}(x,0)+\sum_{j=1}^n\sum_{\substack{y\in \mathbb H_{n-j}\\z\notin \mathbb H_{n-j}\\ y\sim z}}G_\beta^{\mathbb H_{n-j}}(x,y)\beta G_\beta^{\mathbb H_{n-j+1}}(z,0).
\end{equation}
Summing \eqref{eq: second decomp third equation} over $x\in \partial\mathbb H_{n-k}$, we obtain similarly,
\begin{equation}
	\sum_{x\in \partial\mathbb H_{n-k}}G_\beta^{\mathbb H_n}(0,x)\leq \frac{4(n+1)}{\beta}\leq \frac{4(k+1)}{\beta}.
\end{equation}
 This concludes the proof.
\end{proof}
\begin{figure}[!htb]
	\centering
    \includegraphics{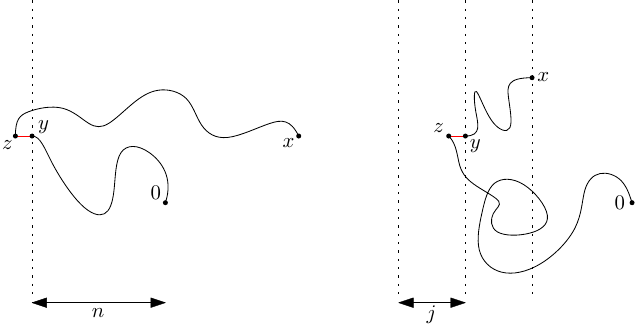}
    \put(-45,140){$\partial\mathbb H_{n-k}$}
    \put(-110,140){$\partial\mathbb H_{n}$}
    \put(-286,140){$\partial\mathbb H_{n}$}
    \caption{On the left, an illustration of the decomposition of the path used in the proof of \eqref{eq: full plane estimate from half plane}. The red edge $yz$ is the earliest edge satisfying that $z$ is a left-most point of the path. On the right, an illustration of the decomposition of the path used in the proof of \eqref{eq: half plane at distance $k$ from half space}.}
    \label{fig: Lemma1}
\end{figure}

We now turn to the estimate of the error amplitude in \eqref{eq:reversed SL} when $\beta<\beta^*$. 
\begin{Lem}[Bounding the error amplitude]\label{lem:bound error}
Fix $d>4$ and $C>\bfC_{\rm RW}$. There exists $K=K(C,d)>0$ such that for every 
$\beta<\beta^*(C)$, and every $n\ge0 $, 
\begin{align}
E_\beta^{\mathbb H_n,\mathbb Z^d}&\le K,\\
\sup\big\{E_\beta^{B,\mathbb Z^d}:B\in \mathcal B\}&\le K.\label{eq:bound error 2 wsaw}
\end{align}
\end{Lem}

\begin{proof} The second inequality follows from the first one (by changing $K$) since for every $B\in \mathcal B$,
\begin{align}\label{eq:bound box by halfspace}
E_\beta^{B,\mathbb Z^d}&\le 2d\max \big\{E_\beta^{\mathbb H_k,\mathbb Z^d}:k\geq 0\}.
\end{align}
For the first inequality, we notice that $E_\beta^{\mathbb H_n,\mathbb Z^d}$ is increasing in $\beta$, so that it is sufficient to prove the bound for $\beta\geq \tfrac{1}{4d}$ (recall that since $C>\bfC_{\rm RW}$, one has $\beta^*\ge\tfrac1{2d}$). The previous lemma (more specifically \eqref{eq: full plane estimate from half plane} and \eqref{eq: sum at distance k from halfplane}) and \eqref{eq:H_beta} give
\begin{align}
E_\beta^{\mathbb H_n,\mathbb Z^d}
&= \sum_{k\ge 0} \sum_{u\in \partial\mathbb H_{n-k}}\sum_{\substack{y \in\mathbb H_n\\ z\notin \mathbb H_n\\ y\sim z}}G_\beta^{\mathbb H_n}(0,u)G_\beta^{\mathbb H_n}(u,y)\beta G_\beta(z,u)\\
&\le
\sum_{k\ge 0}\Big( \sum_{u\in\partial\mathbb H_{n-k}}G_\beta^{\mathbb H_n}(0,u) \Big)\cdot \Big(1+\frac{1}{2d}\Big)\cdot \frac{3C}{(k+1)^{d-2}}
\\&\le \sum_{k\ge 0} \frac{4(k+1)}{\beta} \cdot \Big(1+\frac{1}{2d}\Big)\cdot \frac{3C}{(k+1)^{d-2}}\label{eq: where d>4 is important}
\end{align}
which is finite and depends only on $C,d$ as soon as $d>4$ and $\beta\geq\tfrac1{4d}$. 
\end{proof}

We are now in a position to prove Proposition~\ref{prop:bound phi}.

\begin{proof}[Proof of Proposition~\textup{\ref{prop:bound phi}}]
Summing \eqref{eq:reversed SL} for $S=\mathbb H_n$ and $\Lambda=\mathbb Z^d$ over every $x\in \mathbb Z^d$ gives
\begin{equation}\label{eq: bound phi_n1 wsaw}
\varphi_{\beta}(\mathbb H_n)\chi(\beta)-\lambda \chi(\beta)E_{ \beta}^{\mathbb H_n,\mathbb Z^d}\leq \chi(\beta),
\end{equation}
where we used that $G^{\mathbb H_n}_\beta(0,x)\geq 0$ for all $x\in \mathbb Z^d$. Dividing by $\chi(\beta)$ and using Lemma~\ref{lem:bound error}, we obtain
\begin{equation}
    \varphi_{\beta}(\mathbb H_n)\leq 1+\lambda E_{\beta}^{\mathbb H_n,\mathbb Z^d}\le 1+K\lambda.
   \end{equation}
The same reasoning, with \eqref{eq:reversed SL} applied to $S=B\in \mathcal B$ and $\Lambda=\mathbb Z^d$ yields
\begin{equation}
	\varphi_\beta(B)\leq 1+\lambda E_\beta^{B,\mathbb Z^d}\stackrel{\eqref{eq:bound error 2 wsaw}}\leq 1+K\lambda,
\end{equation}
which concludes the proof.
\end{proof}


\subsection{Control of the gradient}

Proposition \ref{prop:bound phi} implies an $\ell^1$-type bound on $G_\beta^{\mathbb H_n}$ which involves the quantity $\lambda$ and which is  better than \eqref{eq:H_beta}. The following regularity estimate, which will be the goal of this section, will later allow us to convert this $\ell^1$ bound into an improved $\ell^\infty$ bound.

\begin{Prop}[Regularity estimate at mesoscopic scales]\label{prop:regularity}
Fix $d>4$ and $C>\bfC_{\rm RW}$. For every $\eta>0$, there exists $\delta=\delta(\eta,d)\in(0,1/2)$ and $\lambda_0=\lambda_0(\eta,C,d)>0$ such that for every $\lambda<\lambda_0$, every $\beta<\beta^*(C,\lambda)$, every integer $n$ with $\lfloor \delta n\rfloor\geq 6$, every $\Lambda\supset \Lambda_{3n}$, and every $x\in \Lambda\setminus\Lambda_{3n}$, 
\begin{align}
\max\big\{|G_\beta^\Lambda(u,x)-G_\beta^\Lambda(v,x)|:u,v\in\Lambda_{\lfloor\delta n\rfloor}\cap 2\mathbb Z^d\big\}&\le \eta\max\big\{G_\beta^\Lambda(w,x):w\in \Lambda_{3n}\big\}.
\end{align}
\end{Prop}

\begin{Rem}
We will see later that the assumption $u,v\in 2\mathbb Z^d$ is not necessary (see Corollary \ref{cor: improved reg parity-wise}), but we will not need this stronger result in this section.
\end{Rem}
We start with a lemma. Let $\Lambda_n^+:=\{x\in \Lambda_n:x_1>0\}$ and $H:=\{v \in \mathbb Z^d: \: v_1=0\}$.
\begin{Lem}\label{lem:estimate half-space}
Fix $d>4$ and $C>\bfC_{\rm RW}$. Assume that $\lambda\leq \tfrac{1}{2dK}$, where $K=K(C,d)$ is given by Proposition \textup{\ref{prop:bound phi}}. Let $n\geq 12$. For every $v\in \Lambda_k^+$ with $6\le k\le n/2$ and every $\beta<\beta^*(C,\lambda)$, 
\begin{equation}
\sum_{\substack{y \in \Lambda_n^+\\ z\notin \Lambda_n^+\cup H\\ y\sim z}}G_\beta^{\Lambda_n^+}(v,y)\beta \le 4 \Big(\frac {2k}n\Big)^c,
\end{equation}
where $c:=\tfrac{|\log (1-\tfrac{1}{4d^2})|}{2\log 2}$.
\end{Lem}

\begin{proof}
Define $(n_\ell)$ by $n_0:=n$ and then $n_{\ell+1}:=\lfloor (n_\ell-1)/2\rfloor$. We prove by induction that for every $\ell\ge0$ and $v\in \Lambda_{n_\ell}^+$,
\begin{equation}\label{eq: induction}
\sum_{\substack{y \in \Lambda_n^+\\ z\notin \Lambda_n^+\cup H\\ y\sim z}}G_\beta^{\Lambda_n^+}(v,y)\beta \le \Big(1-\frac1{4d^2}\Big)^{\ell}\Big(1+\frac{1}{2d}\Big).
\end{equation}
The case $\ell=0$ follows from Proposition~\ref{prop:bound phi} and from the assumption made on $\lambda$.  Let us transfer the estimate from $\ell$ to $\ell+1$. Fix $v\in \Lambda_{n_{\ell+1}}^+$. Let $B:=\Lambda_{v_1-1}(v)$, which is included in $\Lambda_{n_\ell}^+$ and has one of its faces at distance $1$ from $H$. By symmetry, we have that 
\begin{equation}\label{eq:j1}
\sum_{\substack{r \in B\\ s\notin B\cup H\\ r\sim s}}G_\beta^{B}(v,r)\beta \le \frac{2d-1}{2d}\varphi_\beta(B)\stackrel{\eqref{eq:bound varphi box wsaw}}\le \Big(1-\frac1{2d}\Big)(1+\lambda K)\leq\Big(1-\frac1{2d}\Big)\Big(1+\frac{1}{2d}\Big) =1-\frac{1}{4d^2}.
\end{equation}
Lemma~\ref{Lem: Simon-Lieb WSAW} (applied to $S=B$ and $\Lambda=\Lambda_n^+$) and the induction hypothesis imply that 
\begin{align}
\sum_{\substack{y \in \Lambda_n^+\\ z\notin \Lambda_n^+\cup H\\ y\sim z}}G_\beta^{\Lambda_n^+}(v,y)\beta&\stackrel{\phantom{\eqref{eq:j1}}}\le \sum_{\substack{y \in \Lambda_n^+\\ z\notin \Lambda_n^+\cup H\\ y\sim z}}\Big(\sum_{\substack{r \in B\\ s\notin B\cup H\\ r\sim s}}G_\beta^{B}(v,r)\beta G_\beta^{\Lambda_n^+}(s,y)\Big)\beta \\
&\stackrel{\phantom{\eqref{eq:j1}}}=\sum_{\substack{r \in B\\ s\notin B\cup H\\ r\sim s}}G_\beta^{B}(v,r)\beta\Big(\sum_{\substack{y \in \Lambda_n^+\\ z\notin \Lambda_n^+\cup H\\ y\sim z}}G_\beta^{\Lambda_n^+}(s,y)\beta\Big) \\
&\stackrel{\phantom{\eqref{eq:j1}}}\le \Big(1-\frac1{4d^2}\Big)^{\ell}\Big(1+\frac{1}{2d}\Big)\sum_{\substack{r \in B\\ s\notin B\cup H\\ r\sim s}}G_\beta^{B}(v,r)\beta\\
&\stackrel{\eqref{eq:j1}}\le \Big(1-\frac1{4d^2}\Big)^{\ell+1}\Big(1+\frac{1}{2d}\Big).
\end{align}
This concludes the proof of the induction. 
Now, if $k\leq n/2$, one has\footnote{Indeed, by definition $n_\ell\geq \tfrac{n}{2^\ell}-\sum_{k=0}^{\ell-1}\tfrac{3}{2^k}\ge \tfrac{n}{2^\ell}-6$ for all $\ell\geq 0$. Hence, for $\ell=\lfloor \alpha/2\rfloor$ with $\alpha=\log_2(\tfrac{n}{2k})$,
\begin{equation*}
	n_{\ell}\geq \frac{n}{2^{\alpha/2}}-6=\sqrt{2k}\sqrt{n}-6\geq 2k-6\geq k,
\end{equation*}
where we used that $n\geq 2k$ and $k\geq 6$.} $k\leq n_{\lfloor\alpha/2\rfloor}$ where $\alpha=\log_2(\tfrac{n}{2k})$. Hence, by \eqref{eq: induction}, if $v\in \Lambda_k^+$,
\begin{equation}
	\sum_{\substack{y \in \Lambda_n^+\\ z\notin \Lambda_n^+\cup H\\ y\sim z}}G_\beta^{\Lambda_n^+}(v,y)\beta \le\Big(1-\frac{1}{4d^2}\Big)^{\lfloor\tfrac{1}{2}\log_2(\tfrac{n}{2k})\rfloor}\Big(1+\frac{1}{2d}\Big)\leq \Big(\frac{2k}{n}\Big)^{c}\Big(1+\frac{1}{2d}\Big)^2\leq 4\Big(\frac{2k}{n}\Big)^c,
\end{equation}
where $c=\tfrac{|\log (1-\tfrac{1}{4d^2})|}{2\log 2}$, and where we used that $(1-\tfrac{1}{4d^2})^{-1}\leq 1+\tfrac{1}{2d}\leq 2$ for $d\geq 1$.
\end{proof}

\begin{proof}[Proof of Proposition~\textup{\ref{prop:regularity}}]
It suffices to prove the statement when $u$ and $v$ differ in one coordinate only as the general case follows by summing increments over coordinates. By rotating and translating, we may consider $u=k{\bf e}_1$ and $v=-k{\bf e}_1$ belong to $\Lambda_{\lfloor\delta n \rfloor}\cap (2\mathbb Z^d)$ for $\delta$ to be fixed, and later replace the maximum in $\Lambda_{2n}$ by the maximum in $\Lambda_{3n}\supset \Lambda_{2n+\lfloor \delta n\rfloor}$. The restriction to $(2\mathbb Z^d)$ is technical and comes from the fact that we want $u$ to be the reflection of $v$ with respect to a hyperplane $H\subset \mathbb Z^d$.

Consider the sets $A:=\Lambda_n^+$ and $B:=-\Lambda_n^+$. Applying Lemma~\ref{Lem: Simon-Lieb WSAW} twice as well as Lemma~\ref{lem:bound error} gives
\begin{align}
G_\beta^\Lambda(u,x)&\le \sum_{\substack{y \in A\\ z\notin A\\ y\sim z}}G_\beta^A(u,y)\beta G_\beta^\Lambda(z,x),\\
G_\beta^\Lambda(v,x)&\ge \sum_{\substack{y \in B\\ z\notin B\\ y\sim z}}G_\beta^B(v,y)\beta G_\beta^\Lambda(z,x) -K\lambda\max\{G_\beta^\Lambda(w,x):w\in B\}\label{eq: second inequality first regularity wsaw},
\end{align}
where in \eqref{eq: second inequality first regularity wsaw} we used that 
\begin{align}
	\sum_{u\in B}E^{B,\Lambda}_\beta(u)G_\beta^{\Lambda}(u,x)&\stackrel{\phantom{\eqref{eq:bound error 2 wsaw}}}\leq E_\beta^{B,\Lambda}\cdot\max\{G_\beta^{\Lambda}(w,x):w\in B\}\\&\stackrel{\eqref{eq:bound error 2 wsaw}}\leq K \max\{G_\beta^{\Lambda}(w,x):w\in B\}.
\end{align}
We take the difference and use that when $z\in H$, the corresponding terms in the two sums cancel each other, see Figure \ref{fig:reg}. Assume that $\lambda\leq \tfrac{1}{2dK}$ and $\lfloor \delta n\rfloor\geq 6$. Lemma~\ref{lem:estimate half-space} applied to $k=\lfloor \delta n\rfloor$ and $n$ gives \begin{align}
G_\beta^\Lambda(u,x)-G_\beta^\Lambda(v,x)&\le \sum_{\substack{y \in A\\ z\notin A\cup H\\ y\sim z}}G_\beta^{A}(u,y)\beta G_\beta^\Lambda(z,x)+K\lambda \max\{G_\beta^\Lambda(w,x):w\in B\}\\
&\le \Big(4(2\delta)^c+K\lambda\Big)\max\{G_\beta^\Lambda(w,x):w\in\Lambda_{2n}\}.
\end{align}
The same bound holds if we replace $u$ and $v$. The proof follows by choosing $\delta=\delta(\eta,d)>0$ and $\lambda_0=\lambda_0(\eta,C,d)>0$ small enough. \end{proof}

%

\begin{figure}
	\begin{center}
		\includegraphics{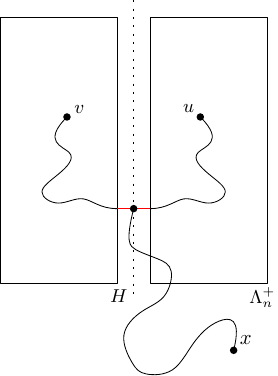}
		\caption{An illustration of the pairing used in the proof of Proposition~\ref{prop:regularity}. Since $u$ (resp. $v$) is close to $H$, a path started from $u$ will most likely touch $H$ if it exits $\Lambda_n^+$.}
		\label{fig:reg}
	\end{center}
\end{figure}

\subsection{Wrapping up the proof of Theorem~\ref{thm:mainwsaw}}\label{sec:2.3}

We start by showing how to use Proposition \ref{prop:regularity} to turn the improved $\ell^1$ estimate of $G^{\mathbb H_n}_\beta$ given by Proposition \ref{prop:bound phi} into an improved $\ell^\infty$ bound.
\begin{Prop}[Improving the bound on $G_\beta^{\mathbb H_n}$]\label{prop: improve l inf bound} Fix $d>4$. There exist $C,\lambda_0>0$ such that, for every $\lambda<\lambda_0$, and every $\beta<\beta^*(C,\lambda)$,
\begin{equation}
	G_{\beta}^{\mathbb H_n}(0,x)\leq \frac{C}{2(1\vee n)^{d-1}}\qquad \forall n\geq 0 ,\forall x\in \partial \mathbb H_n.
\end{equation}
\end{Prop}

\begin{proof} Let $\eta>0$ and $C>\bfC_{\rm RW}$ to be fixed. By monotonicity of $G_{\beta}^{\mathbb H_n}(0,x)$ in $\beta$, it suffices to prove the result for $\beta\in [\tfrac{1}{4d},\beta^*)$. Let $\delta=\delta(\eta,d)$ and $\lambda_0=\lambda_0(\eta,C,d)$ be provided by Proposition~\ref{prop:regularity}. Assume that $\tfrac{\delta n}{6}$ is an integer (otherwise simply round the number) and that $\tfrac{\delta n}{6}\geq 6$. Set \begin{equation}
V_n=V_n(\delta):=\Big\{x\in\Lambda_{\delta n/6}:x_1=0,\:x_j\text{ even for }j\ge2\Big\}.
\end{equation} 
Proposition~\ref{prop:regularity} applied to $n/6$ gives that for every $\beta<\beta^*$, every $x\in \partial\mathbb H_n$ and $y\in V_n$,
\begin{align}\label{eq:h1}
G_{\beta}^{\mathbb H_n}(0,x-y)&=G_{\beta}^{\mathbb H_n}(y,x)\\
&\ge G_{\beta}^{\mathbb H_n}(0,x)-\eta\max\{G_{\beta}^{\mathbb H_n}(v,x):v\in\Lambda_{n/2}\}.
\end{align}
Fix $x\in \partial \mathbb H_n$. Averaging the last displayed equation over $y\in V_n$ and using that $\beta\ge \tfrac1{4d}$ gives
\begin{align}\label{eq:h2}
\frac{8d}{|V_n|}\stackrel{\phantom{\eqref{eq:H_beta}}}\geq\frac{4d(1+(2d)^{-1})}{|V_n|}\stackrel{\eqref{eq:H_beta}}\ge \frac{\varphi_{\beta}(\mathbb H_n)}{\beta{|V_n|}}\geq \frac{1}{|V_n|}\sum_{y\in V_n}G_{\beta}^{\mathbb H_n}(0,x-y)\ge G_{\beta}^{\mathbb H_n}(0,x)-\eta\cdot \frac{C}{(n/2)^{d-1}},
\end{align}
where in the last inequality we used \eqref{eq:H_beta-''} to argue that
\begin{equation}
	\max\{G_{\beta}^{\mathbb H_n}(v,x):v\in\Lambda_{n/2}\}\leq \frac{C}{(n/2)^{d-1}}.
\end{equation}
If $\eta$ is chosen so that $\eta\le 1/2^{d+1}$, then $C=C(\delta,d)$ is chosen large enough, we find that
\begin{equation}
\sup_{x\in \partial\mathbb H_n}G_{\beta}^{\mathbb H_n}(0,x)\leq\frac{C}{2n^{d-1}}
\end{equation}
provided $n$ is large enough (i.e.\ $\tfrac{\delta n}{6}\geq 6$). 

To treat the small values of $n$, we observe that for all $n\geq 0$, for all $x\in \partial \mathbb H_n$, $G_\beta^{\mathbb H_n}(0,x)\le \tfrac{1}{\beta}\varphi_\beta(\mathbb H_n)\le \tfrac{1}{\beta}(1+\tfrac{1}{2d})\leq 8d$. Thus, if we additionally require that $C\geq 16d(36\delta^{-1})^{d-1}$, we ensure that for all $0\leq n\leq 36\delta^{-1}$,
\begin{equation}
	\sup_{x\in \partial\mathbb H_n}G^{\mathbb H_n}(0,x)\leq \frac{C}{2(1\vee n)^{d-1}}.
\end{equation}
This concludes the proof.
\end{proof}
\begin{Prop}\label{prop: theorem beta start = beta_c}
Fix $d>4$. There exist $C,\lambda_0>0$, and $K=K(C,d)>0$ such that for every $\lambda<\lambda_0$,
\begin{align}
G_{\beta_c}(0,x)&\le \frac{3C}{(1\vee|x|)^{d-2}}&\forall x\in \mathbb Z^d,\label{eq:upper bound below L}\\
G_{\beta_c}^{\mathbb H}(0,x)&\le \frac{C}{(1\vee|x_1|)^{d-1}}&\forall x\in \mathbb H
\label{eq:upper half-spacebound below L},\\
\label{eq: bound phi_n final prop}
\varphi_{\beta_c}(\mathbb H_n)&\le 1+K\lambda &\forall n\ge0,\\
\sup\{\varphi_{\beta_c}(B):B\in \mathcal B\}&\le 1+K\lambda, &\\
E^{\mathbb H_n,\mathbb Z^d}_{\beta_c}&\leq K &\forall n\geq 0,
\\
\sup\big\{E^{B,\mathbb Z^d}_{\beta_c}:B\in \mathcal B\big\}&\leq K. &\label{eq:general bound error last prop}
\end{align} 
\end{Prop}
\begin{proof} We let $C,\lambda_0$ be given by Proposition \ref{prop: improve l inf bound}. This choice gives (still by Proposition \ref{prop: improve l inf bound}) that for every $\lambda<\lambda_0$ and every $\beta<\beta^*(C,\lambda)$,
\begin{equation}\label{eq:proof improved bound below beta*}
	G_\beta^{\mathbb H_n}(0,x)\leq \frac{C}{2(1\vee n)^{d-1}} \qquad \forall n\geq 0, \forall x\in \partial \mathbb H_n.
\end{equation}
Let $K=K(C,d)$ be given by Proposition \ref{prop:bound phi}.
We potentially decrease the value of $\lambda_0$ (which does not affect the value of $C$) and require that $K\lambda_0<\tfrac{1}{4d}$. Proposition \ref{prop:bound phi} implies that for every $\lambda<\lambda_0$ and every $\beta<\beta^*(C,\lambda)$,
\begin{equation}\label{eq:proof improved bound below beta* l1}
	\varphi_\beta(\mathbb H_n)\leq 1+K\lambda\leq 1+\frac{1}{4d} \qquad \forall n\geq 0.
\end{equation}
The monotone convergence theorem implies that \eqref{eq:proof improved bound below beta*}--\eqref{eq:proof improved bound below beta* l1} still hold at $\beta^*$. We will use \eqref{eq:proof improved bound below beta*}--\eqref{eq:proof improved bound below beta* l1}  to prove that for this choice of $C,\lambda_0$: for all $\lambda<\lambda_0$,
\begin{equation}
	\beta^*=\beta^*(C,\lambda)=\beta_c.
\end{equation}

Let $\lambda<\lambda_0$. Assume by contradiction that $\beta^*=\beta^*(C,\lambda)<\beta_c$. Exponential decay of the two-point function below $\beta_c$ implies that for every $\beta<\beta_c$, there exists $c(\beta),C(\beta)>0$ such that for all $n\geq 0$, for all $x\in \partial \mathbb H_n$,
\begin{equation}\label{eq: exponential decay wsaw}
	G^{\mathbb H_n}_\beta(0,x)\leq C(\beta)e^{-c(\beta)|x|}.
\end{equation}
Let $\beta^{**}$ be any number in $(\beta^*,\beta_c)$. By monotonicity of all the quantities in $\beta$ and \eqref{eq: exponential decay wsaw} applied at $\beta^{**}$, there exists $N\ge 0$ such that for every $\beta<\beta^{**}$, one has
\begin{equation}\label{eq: improved bound large n}
	\varphi_{\beta}(\mathbb H_n)<1+\frac{1}{2d}\quad \forall n\geq N, \qquad G_\beta^{\mathbb H_n}(0,x)<\frac{C}{(1\vee n)^{d-1}}\quad \forall n\geq 0, \forall x\in \partial \mathbb H_n\setminus \Lambda_N.
\end{equation}
Now, \eqref{eq: improved bound large n}, the validity of \eqref{eq:proof improved bound below beta*}--\eqref{eq:proof improved bound below beta* l1} at $\beta^*$, together with the continuity (below $\beta_c$) of the maps $\beta\mapsto \max\{\varphi_\beta(\mathbb H_n):0\leq n \leq N-1\}$ and $\beta\mapsto \max\{G_\beta^{\mathbb H_n}(0,x):x\in \Lambda_{N}\cap \partial \mathbb H_n, \: n\geq 1\}$ yield the existence of $\beta\in(\beta^*,\beta^{**})$ such that
\begin{equation}
	\varphi_\beta(\mathbb H_n)<1+\frac{1}{2d} \quad \forall n\geq 0, \qquad G_\beta^{\mathbb H_n}(0,x)<\frac{C}{(1\vee n)^{d-1}} \qquad \forall n\geq 0, \forall x\in \partial \mathbb H_n.
\end{equation}
This clearly contradicts the definition of $\beta^*$ and therefore implies that $\beta^*=\beta_c$.

For the proofs of \eqref{eq:upper bound below L}--\eqref{eq:general bound error last prop}, we use Lemmata \ref{lem: full plane 2pt function out of halfplane one}--\ref{lem:bound error} and Proposition \ref{prop:bound phi} to show that the bounds hold for $\beta<\beta_c$. We extend these bounds at $\beta_c$ using once again the monotone convergence theorem.
\end{proof}

\begin{Rem}\label{rem: beta_c}
The bound $2d\beta_c=\varphi_{\beta_c}(\{0\})\le1+K\lambda$ gives $\beta_c\le \tfrac1{2d}(1+K\lambda)$, which complements the bound $\beta_c\geq\tfrac1{2d}$.
\end{Rem}
We are now in a position to prove Theorem~\ref{thm:mainwsaw}.
\begin{proof}[Proof of Theorem~\textup{\ref{thm:mainwsaw}}] Let $C,\lambda_0$ be given by Proposition \ref{prop: theorem beta start = beta_c} and let $\lambda<\lambda_0$. Let $A=A(d)\geq 1$ be such that 
\begin{equation}\label{eq:def A wsaw}
t^{d-1}\leq Ae^{t}\qquad   \forall t\geq 1.
\end{equation}
\paragraph{Proof of \eqref{eq:bound full plane}.}
First, if $\beta\leq \beta_c$,
Proposition \ref{prop: theorem beta start = beta_c} implies that for $x\in \mathbb Z^d$ with $|x|\le L_\beta$,
\begin{equation}\label{eq: proof nearcritical full space 0 wsaw}
	G_\beta(0,x)\leq \frac{3C}{(1\vee|x|)^{d-2}}\leq \frac{3e^2C}{(1\vee|x|)^{d-2}}e^{-2|x|/L_\beta}.
\end{equation}
We now turn to the case of $\beta<\beta_c$ and $|x|>L_\beta$ (recall that $L_{\beta_c}=\infty$). Iterating \eqref{eq:SL} $k:=\lfloor |x|/(L_\beta+1)\rfloor-1$ times with $S$ being translates of $\Lambda_{L_\beta}$ and $\Lambda=\mathbb Z^d$, we get that\footnote{Recall that by definition of $L_\beta$, one has $\varphi_\beta(\Lambda_{L_\beta})\leq e^{-2}$.} 
\begin{align}\label{eq:iterate}
G_\beta(0,x)&\stackrel{\phantom{\eqref{eq:upper bound below L}}}\le \varphi_\beta(\Lambda_{L_\beta})^k \max\big\{G_\beta(y,x):y\notin\Lambda_{L_\beta}(x)\big\}\\&\stackrel{\eqref{eq:upper bound below L}}\leq e^{-2k}\frac{3C}{L_\beta^{d-2}}\stackrel{\phantom{\eqref{eq:upper bound below L}}}\leq \frac{3e^4C}{L_\beta^{d-2}}e^{-2|x|/L_\beta}
\stackrel{\eqref{eq:def A wsaw}}\leq \frac{3e^4AC}{|x|^{d-2}}e^{-|x|/L_\beta}.\label{eq: proof nearcritical full space wsaw}
\end{align}
\paragraph{Proof of \eqref{eq:bound half plane}.} Again, if $\beta\leq \beta_c$, Proposition \ref{prop: theorem beta start = beta_c} implies that for $x\in \mathbb H$ with $x_1\leq L_\beta$,
\begin{equation}\label{eq: proof nearcritical half space 0 wsaw}
	G_\beta^{\mathbb H}(0,x)\leq \frac{C}{(1\vee |x_1|)^{d-1}}\leq \frac{e^2C}{(1\vee|x_1|)^{d-1}}e^{-2|x_1|/L_\beta}.
\end{equation}
Turning to the case $\beta<\beta_c$ and $|x_1|>L_\beta$, iterating \eqref{eq:SL} $\ell:=\lfloor |x_1|/(L_\beta+1)\rfloor-1$ times with $S$ being translates of $\Lambda_{L_\beta}$ (starting with $\Lambda_{L_\beta}(x)$) and $\Lambda=\mathbb H$, we obtain
\begin{align}
	G_\beta^{\mathbb H}(0,x)=G_\beta^{\mathbb H}(x,0)\leq \varphi_\beta(\Lambda_{L_\beta})^{\ell}\max\{G_\beta^{\mathbb H}(0,y): y_1> L_\beta\}&\stackrel{\eqref{eq:upper half-spacebound below L}}\leq \frac{e^4C}{L_\beta^{d-1}}e^{-2|x_1|/L_\beta}
	\\&\stackrel{\eqref{eq:def A wsaw}}\leq \frac{e^4AC}{|x_1|^{d-1}}e^{-|x_1|/L_\beta} \label{eq: proof nearcritical half space wsaw}
\end{align}
Gathering \eqref{eq: proof nearcritical full space 0 wsaw}, \eqref{eq: proof nearcritical full space wsaw}, \eqref{eq: proof nearcritical half space 0 wsaw}, and \eqref{eq: proof nearcritical half space wsaw}, we obtained: for all $\beta\leq \beta_c$, for all $x\in \mathbb Z^d$,
\begin{align}
G_\beta(0,x)&\le \frac{3e^4AC}{(1\vee|x|)^{d-2}}\exp(-|x|/L_\beta),\\
G_\beta^\mathbb H(0,x)&\le \frac{3e^4AC}{(1\vee|x_1|)^{d-1}}\exp(-|x_1|/L_\beta).
\end{align}
This concludes the proof.
\end{proof}

\section{Proof of Theorem~\ref{thm:main2wsaw}}

For $\varepsilon\in (0,1)$, introduce the quantity
\begin{equation}
L_\beta(\varepsilon):=\inf\{k\ge1:\varphi_\beta(\Lambda_k)\le 1-\varepsilon\}.
\end{equation}
In particular, one has $L_\beta=L_\beta(1-e^{-2})$.
This new correlation length only differs from $L_\beta$ by a constant, as stated in the next lemma, but it will be more convenient for the rest of the proof. A similar result is derived in the context of (spread-out) Bernoulli percolation in our companion paper, see \cite[Lemma~4.3]{DumPan24Perco}.

\begin{Lem}[Comparison between $L_\beta$ and $L_\beta(\varepsilon)$]\label{lem: comparison of diff l} Fix $d>4$ and $\varepsilon\in(0,1-e^{-2})$. Let $\lambda_0$ be given by Proposition \textup{\ref{prop: theorem beta start = beta_c}}. There exists $C(\varepsilon,\lambda_0,d)\in (0,\infty)$ such that, for every $\lambda<\lambda_0$, and every $\beta<\beta_c$,

\begin{equation}
	L_\beta(\varepsilon)\leq L_\beta\leq C(\varepsilon,\lambda_0)L_\beta(\varepsilon).
\end{equation}
\end{Lem}

\begin{proof} Observe that if $(\varepsilon,\varepsilon')\in (\mathbb R^+)^2$ with $\varepsilon\leq \varepsilon'$, then 
\begin{equation}\{k\geq 1: \varphi_\beta(\Lambda_k)\leq \varepsilon\}\subset \{k\geq 1: \varphi_\beta(\Lambda_k)\leq \varepsilon'\}.\end{equation}
 Hence, the map $\varepsilon\mapsto L_\beta(\varepsilon)$ is increasing. This observation, together with the fact that $L_\beta=L_\beta(1-e^{-2})$ and $\varepsilon \in (0,1-e^{-2})$, implies the first inequality. 

We turn to the second one.Let $k\geq 1$ to be chosen large enough. Assume by contradiction that $L_\beta>k(L_\beta(\varepsilon)+1)$. Similarly to \eqref{eq:iterate}, if we iterate \eqref{eq:SL} $k$ times with $S$ being translates of $\Lambda_{L_\beta(\varepsilon)}$ and $\Lambda=\Lambda_{L_\beta-1}$, we obtain that: for every $x\in \partial \Lambda_{L_\beta-1}$,
\begin{equation}
	G_\beta^{\Lambda}(0,x)\beta\leq \sum_{\substack{u_1\in \Lambda_{L_\beta(\varepsilon)}\\v_1\notin \Lambda_{L_\beta(\varepsilon)}\\ u_1\sim v_1}}G_\beta^{\Lambda_{L_\beta(\varepsilon)}}(0,u_1)\beta\ldots\sum_{\substack{u_k\in \Lambda_{L_\beta(\varepsilon)}(v_{k-1})\\v_k\notin \Lambda_{L_\beta(\varepsilon)}(v_{k-1})\\ u_k\sim v_k}}G_\beta^{\Lambda_{L_\beta(\varepsilon)}(v_{k-1})}(v_{k-1},u_k)\beta G_\beta^\Lambda(v_k,x)\beta.
	\end{equation}
Summing the above displayed equation over $x\in \partial\Lambda_{L_\beta-1}$ and $y\sim x$ with $y\notin \Lambda_{L_\beta-1}$ gives
\begin{equation}
\varphi_\beta(\Lambda_{L_\beta-1})\le \varphi_\beta(\Lambda_{L_\beta(\varepsilon)})^k\max\{\varphi_\beta(\Lambda_{L_\beta-1}(x)):x\in \Lambda_{L_\beta-1}\}\le (1-\varepsilon)^k(1+K\lambda),
\end{equation}
where $K$ is given by Proposition \ref{prop: theorem beta start = beta_c}. Choosing $k=k(\varepsilon,\lambda_0,d)$ large enough so that $(1-\varepsilon)^k(1+K\lambda_0)<e^{-2}$ contradicts the above display by definition of $L_\beta$. This implies that 
\begin{equation}L_\beta\le k(L_\beta(\varepsilon)+1)\leq 2kL_\beta(\varepsilon),\end{equation} where we used that $L_\beta(\varepsilon)\geq 1$.
We obtained that $L_\beta \le C(\varepsilon,\lambda_0,d) L_\beta(\varepsilon)$, where we set $C(\varepsilon,\lambda_0,d):=2k<\infty$. This concludes the proof.
\end{proof}
Let $n\geq 1$ and $\varepsilon\in(0,1-e^{-2})$. Introduce the set
\begin{equation}\label{eq: def An}
	A_n:=\{x\in \mathbb Z^d: x_1=|x|=n\}.
\end{equation}
The assumption $n<L_\beta(\varepsilon)$ implies an $\ell^1$-type lower bound on the half-space two-point function. Indeed, if $n< L_\beta(\varepsilon)$, one has
\begin{equation}\label{eq:averaged lower bound half-space}
	\frac{1}{|A_n|}\sum_{x\in A_n}G_\beta^{\mathbb H}(0,x)\geq \frac{1}{|A_n|}\sum_{x\in \partial \Lambda_n:\: x_1=n}G_\beta^{\Lambda_n}(0,x)\geq \frac{\varphi_\beta(\Lambda_n)}{2d\beta|A_n|}\geq \frac{(1-\varepsilon)}{2d\beta|A_n|}\geq \frac{c_0}{n^{d-1}},
\end{equation}
where $c_0=c_0(d)>0$. The core of this section will be to turn this averaged estimate into a point-wise estimate for the half-space two-point function. The corresponding lower bound for the full-space two-point function will follow thanks to Lemma \ref{Lem: Simon-Lieb WSAW}. The proof is organised in two steps: we begin by proving a regularity estimate, and then we use it to get the theorem.

\subsection{A Harnack-type estimate}

We start with another regularity estimate relating the minimum to the maximum of the two-point function in a domain.

\begin{Prop}[Harnack-type estimate at macroscopic scales]\label{prop:regularity2}
Fix $d>4$ and $\alpha>0$. There exists $C_{\rm RW}=C_{\rm RW}(\alpha,d)>0$ and for every $\eta>0$, there exist $\lambda_0=\lambda_0(\eta,\alpha,d)$, $\varepsilon_0=\varepsilon_0(\eta,d)>0$ small enough and $N_0=N_0(\eta,\alpha,d)$ large enough such that the following holds. For every $\lambda<\lambda_0$,  every $\varepsilon<\varepsilon_0$,  every $N_0\leq n\le 6L_\beta(\varepsilon)$,  every $\tfrac1{2d}\le\beta\le \beta_c$, every $\Lambda\supset \Lambda_{(1+\alpha)n}$, and every $x\in \Lambda\setminus \Lambda_{(1+\alpha)n}$,
\begin{align}
\max\{G_\beta^\Lambda(u,x):u\in \Lambda_n\}&\le C_{\rm RW}\min\{G_\beta^\Lambda(u,x):u\in \Lambda_n\}+\eta \max\{G_\beta^\Lambda(u,x):u\in \Lambda_{(1+\alpha)n}\}.\end{align}
\end{Prop}

In what follows, we will use Proposition \ref{prop:regularity2} with $\alpha=\tfrac{1}{12}$ (see the proof of Lemma \ref{lem:lower below half-space}) and $\alpha=1$ (see the proof of Theorem \ref{thm:main2wsaw}). We will derive Proposition \ref{prop:regularity2} using classical random walk estimates that are derived in our companion paper \cite{DumPan24Perco}.
We start with a lemma which is useful to go around the parity assumption in Proposition~\ref{prop:regularity}.

\begin{Lem}\label{lem: remove parity}
Fix $d>4$, $\eta>0$, and $\ell\geq 1$. There exist $\lambda_0=\lambda_0(\ell,\eta,d)$ and $L=L(\ell,\eta,d)$ such that the following holds. For every $\lambda<\lambda_0$, every $\tfrac1{2d}\le \beta\le \beta_c$, every set $\Lambda\supset \Lambda_{L}$, and every $x\in \Lambda\setminus\Lambda_L$,
\begin{equation}
\max\big\{|G_\beta^\Lambda(u,x)-G_\beta^\Lambda(v,x)|:u,v\in\Lambda_\ell \big\}\le \eta \max\{G_\beta^\Lambda(w,x):w\in \Lambda_L\}.
\end{equation}
\end{Lem}

\begin{proof} 

Fix $\eta>0$ and $\ell\geq 1$. If $z\in \mathbb Z^d$, we let $\mathbb P_z$ be the law of the simple random walk $(X_k)_{k\geq 0}$ started at $z$, i.e.
\begin{equation}\label{eq: step distribution srw}
\mathbb P_z[X_1=y]:=\tfrac1{2d}\mathds{1}_{y\sim z}, \qquad (y\in \mathbb Z^d).
\end{equation}
In the perspective of importing random walk estimates from \cite{DumPan24Perco}, we observe that--- following the notations of \cite[Appendix~A.2]{DumPan24Perco}--- the step distribution of \eqref{eq: step distribution srw} belongs to $\mathcal P_1$.

Let $L\geq \ell$ to be fixed. Fix $\Lambda\supset \Lambda_L$. Define $\tau=\tau_{L-1}$ to be the hitting time of $\partial\Lambda_{L-1}$.

We may couple the simple random walks started at $u$ and $v$ in such a way that they merge with high probability after $\approx \ell^2$ steps. If $L$ is large enough, this implies that the random walks merge before $\tau$. This classical estimate is presented in Proposition \ref{prop: estimates srw}. As a consequence of this result, there exists $L=L(\ell,\eta,d)>0$ such that, for every $u,v\in \Lambda_\ell$, and every $x\in \Lambda\setminus\Lambda_L$,
\begin{equation}\label{eq:new interm0}
	\Big|\mathbb E_u[G_\beta^{\Lambda}(X_\tau,x)]-\mathbb E_v[G_\beta^{\Lambda}(X_\tau,x)]\Big|\leq \frac{\eta}{4}\max\{G_\beta^{\Lambda}(w,x): w\in \Lambda_{L-1}\}.
\end{equation}
 Thanks to \cite[Corollary~A.6]{DumPan24Perco} (applied to $\eta/8$ and $n=A=L-2$ and $m=1$), we find $T,\varepsilon>0$ (which depend on $L,\eta,d$) satisfying $(1+\varepsilon)^T\leq 2$, for every $\varphi\in[1-\varepsilon,1+\varepsilon]$, and every $u\in \Lambda_\ell$,
 \begin{equation}
 	\Big|\mathbb E_u[\varphi^{\tau\wedge T}G_\beta^{\Lambda}(X_{\tau\wedge T},x)]-\mathbb E_u[G_\beta^{\Lambda}(X_\tau,x)]\Big|\leq \frac{\eta}{4}\max\{G_\beta^{\Lambda}(w,x): w\in \Lambda_{L}\}.
 \end{equation}
 Combining the previously displayed equation with \eqref{eq:new interm0} yields, for every $u,v\in \Lambda_\ell$,
\begin{equation}\label{eq:new interm1}
	\Big|\mathbb E_u[\varphi^{\tau\wedge T}G_\beta^{\Lambda}(X_{\tau\wedge T},x)]-\mathbb E_v[\varphi^{\tau\wedge T}G_\beta^{\Lambda}(X_{\tau\wedge T},x)]\Big|\leq \frac{\eta}{2}\max\{G_\beta^{\Lambda}(w,x): w\in \Lambda_{L}\}.
\end{equation}

From now on, we let $\lambda_0$ small enough be given by Proposition \ref{prop: theorem beta start = beta_c} and (to the cost of diminishing $\lambda_0$) we additionally require that $K\lambda_0\le \varepsilon$ where $K$ is given by the same proposition. Let $\lambda<\lambda_0$ and $\beta\le\beta_c$. The assumption $\beta\ge \tfrac1{2d}$ and Proposition~\ref{prop:  theorem beta start = beta_c} give
\begin{equation}
1\le 2d\beta=\varphi_\beta(\{0\})\le 1+K\lambda\le 1+\varepsilon.
\end{equation}
Introduce the shorthand notation $\varphi:=\varphi_\beta(\{0\})$. Fix $u,v\in \Lambda_\ell$ and $x\in \Lambda\setminus \Lambda_L$.
Iterating the two bounds of Lemma~\ref{Lem: Simon-Lieb WSAW} (with $S$ a singleton) until time $\tau\wedge T$ and using Lemma~\ref{lem:bound error} gives
\begin{align}
G_\beta^\Lambda(u,x)&\stackrel{\eqref{eq:SL}}\le  \mathbb E_{u}[\varphi^{\tau\wedge T} G_\beta^\Lambda(X_{\tau\wedge T},x)],\label{eq:hea1}\\
G_\beta^\Lambda(v,x)&\stackrel{\eqref{eq:reversed SL}}\ge \mathbb E_{v}[\varphi^{\tau\wedge T} G_\beta^\Lambda(X_{\tau\wedge T},x)]-K\lambda \,\mathbb E_v\Big[\sum_{s=0}^{\tau\wedge T}\varphi^s\Big]\max\{G_\beta^\Lambda(w,x):w\in \Lambda_{L}\}\\
&\stackrel{\phantom{\eqref{eq:reversed SL}}}\ge \mathbb E_v[\varphi^{\tau\wedge T} G_\beta^\Lambda(X_{\tau\wedge T},x)]-2K\lambda \mathbb E_v[\tau+1]\max\{G_\beta^\Lambda(w,x):w\in \Lambda_{L}\}\label{eq:hda1}.
\end{align}
Using \cite[Proposition~A.5]{DumPan24Perco} (with $n=L-2$ and $m=1$), we find that
\begin{equation}\label{eq:new interm2}
	\mathbb E_v[\tau+1]\leq 9dL^2+1.
\end{equation}
Taking the difference of \eqref{eq:hea1} and \eqref{eq:hda1} and using \eqref{eq:new interm2} gives
\begin{align}\notag
	G_\beta^{\Lambda}(u,x)-G_\beta^{\Lambda}(v,x)&\stackrel{\phantom{\eqref{eq:new interm1}}}\leq \mathbb E_{u}[\varphi^{\tau\wedge T} G_\beta^\Lambda(X_{\tau\wedge T},x)]-\mathbb E_{v}[\varphi^{\tau\wedge T} G_\beta^\Lambda(X_{\tau\wedge T},x)] \notag
	\\&\qquad +2K(9dL^2+1)\lambda \max\{G_\beta^\Lambda(w,x):w\in \Lambda_{L}\}
	\\&\stackrel{\eqref{eq:new interm1}}\leq \Big(\frac{\eta}{2} +2K(9dL^2+1)\lambda \Big)\max\{G_\beta^\Lambda(w,x):w\in \Lambda_{L}\}.
\end{align}
A similar bound holds with $u$ and $v$ replaced.
The proof follows by choosing $\lambda_0$ small enough so that $2K(9dL^2+1)\lambda_0\leq \eta/2$.
\end{proof}

\begin{Cor}[Regularity estimate at mesoscopic scales without the parity assumption]\label{cor: improved reg parity-wise}	Fix $d>4$ and $\eta>0$. There exist $\delta=\delta(\eta,d)\in(0,1/2)$ and $\lambda_0=\lambda_0(\eta,d)>0$, such that for every $\lambda<\lambda_0$, every $\tfrac{1}{2d}\le\beta\leq \beta_c$, every integer $n$ with $\lfloor n\delta\rfloor\geq 6$, every $\Lambda\supset \Lambda_{3n}$, and every $x\in \Lambda\setminus\Lambda_{3n}$, 
\begin{align}
\max\big\{|G_\beta^\Lambda(u,x)-G_\beta^\Lambda(v,x)|:u,v\in\Lambda_{\lfloor\delta n \rfloor}\big\}&\le \eta\max\big\{G_\beta^\Lambda(w,x):w\in \Lambda_{3n}\big\}.
\end{align}
\end{Cor}
\begin{proof} By continuity, it is sufficient to prove the result for $\beta<\beta_c$.
Take $\lambda_0$ small enough such that Proposition \ref{prop: theorem beta start = beta_c} holds, Proposition \ref{prop:regularity} holds with $\eta/3$, and Lemma \ref{lem: remove parity} holds with $\eta/3$ with $\ell=1$. Let also $\delta$ be given by Proposition \ref{prop:regularity}, and $L$ be given by Lemma \ref{lem: remove parity} for this choice of parameters. By Proposition \ref{prop:regularity}, if $n\geq 6\delta^{-1}$,
\begin{equation}
	\max\big\{|G_\beta^\Lambda(u,x)-G_\beta^\Lambda(v,x)|:u,v\in\Lambda_{\lfloor\delta n\rfloor}\cap 2\mathbb Z^d\big\}\le \frac{\eta}{3}\max\big\{G_\beta^\Lambda(w,x):w\in \Lambda_{3n}\big\}.
\end{equation}
Take $u,v\in \Lambda_{\lfloor \delta n\rfloor}$. There exist $u'$ and $v'$ such that $u',v'\in \Lambda_{\lfloor\delta n\rfloor}\cap 2\mathbb Z^d$, and $|u-u'|,|v-v'|\leq 1$. Assuming that $n\geq L$, Lemma \ref{lem: remove parity} gives
\begin{equation}
	|G_\beta^{\Lambda}(u,x)-G_\beta^{\Lambda}(u',x)|\leq \frac{\eta}{3}\max\{G_\beta^\Lambda(w,x):x\in \Lambda_L+u\}\leq \frac{\eta}{3}\max\{G_\beta^\Lambda(w,x):x\in \Lambda_{3n}\},
\end{equation}
and a similar bound holds for $|G_\beta^{\Lambda}(v,x)-G_\beta^{\Lambda}(v',x)|$. 
Combining the two previously displayed equations gives that for every $n\geq (6\delta^{-1}\vee L)$,
\begin{equation}
	|G_\beta^\Lambda(u,x)-G_\beta^\Lambda(v,x)|\leq \eta \max\{G_\beta^{\Lambda}(w,x):w\in \Lambda_{3n}\}.
\end{equation}
The proof follows readily from decreasing $\delta$ so that $6\delta^{-1}\geq L$.
	\end{proof}

%
We now turn to the proof of Proposition~\ref{prop:regularity2}. The proof follows a similar approach to that of Lemma \ref{lem: remove parity}. As $u$ and $v$ are at a distance approximately $n$ of each other, it is not possible to couple the random walks in such a way that they merge before exiting the box $\Lambda_{(1+\alpha)n}$. However, since we are looking for an up-to-constant comparison of the two-point functions, we can use a Harnack-type estimate to conclude. As we do not want this estimate to depend on $n$,  we replace the simple random walk by a well-chosen rescaled random walk defined as follows: if $m\geq 1$, consider the random walk $(X_k')_{k\geq 0}$ defined by the following step distribution $\mu=\mu_{m,\beta}$: for $u\in \Lambda_n$ and $v\in \mathbb Z^d$,
\begin{equation}\label{eq:step distribution non srw}
\mathbb P_u'[X_1'=v]=\mu(v-u):=\frac{\mathds{1}_{v\notin \Lambda_{m-1}(u)}}{\varphi_\beta(\Lambda_{m-1})}\sum_{\substack{w\in \Lambda_{m-1}(u)\\ w\sim v}}G_\beta^{\Lambda_{m-1}(u)}(u,w)\beta.
\end{equation}
We observe that--- following the notations of \cite[Appendix~A.2]{DumPan24Perco}--- the step distribution of \eqref{eq:step distribution non srw} belongs to $\mathcal P_m$. This allows us to import the following tool.

 \begin{Prop}[{\hspace{1pt}\cite[Proposition~A.7]{DumPan24Perco}}]\label{prop:uniform Harnack} Let $d>4$, $\alpha>0$, and $\eta>0$. There exists $C_{\rm RW}=C_{\rm RW}(\alpha,d)>0$ and $N_1=N_1(\eta,\alpha,d)>0$ such that the following holds. For every $n,m\geq 1$ satisfying $\tfrac{n}{m}\geq N_1$, every $f:\mathbb Z^d\rightarrow \mathbb R^+$, and every $u,v\in \Lambda_n$,
\begin{multline}
	\mathbb E_{u}'[f(X_\tau')]\leq C_{\rm RW}\mathbb E_{v}'[f(X_\tau')] +\eta\max\{f(w):w\in \Lambda_{(1+\alpha)n+2m}\}\\+2C_{\rm RW}\max\Big\{|f(w)-f(w')|: w,w'\in \Lambda_{3m(n/m)^{1/10}}(z), \: z\in \partial \Lambda_{(1+\alpha)n}\Big\},
\end{multline}
where $\tau:=\inf \{k\geq 0: X_k'\notin \Lambda_{(1+\alpha)n}\}$.
\end{Prop}

\begin{proof}[Proof of Proposition~\textup{\ref{prop:regularity2}}]Let $\alpha,\eta>0$. Set $M:=\lfloor \tfrac{\alpha}{6} n\rfloor$, and let $C_{\rm RW}$ and $N_1$ be given by Proposition \ref{prop:uniform Harnack} applied to $\tfrac{\alpha}{6}$ and $\tfrac{\eta}{8}$.
Let $\delta$ and $\lambda_0$ be given by Corollary \ref{cor: improved reg parity-wise} with $\frac{\eta}{4C_{\rm RW}}$. To the cost of diminishing $\delta$, we additionally assume that
\begin{equation}\label{eq:take delta smaller}
	\delta\leq (N_1\alpha)^{-1}.
\end{equation}
Let $\tau$ be the hitting time of $\mathbb Z^d\setminus\Lambda_{n+M}$. If $m=\lfloor (\delta\alpha/36)^{10/9}n\rfloor$, then $n/m\geq N_1$ by \eqref{eq:take delta smaller}. Hence, Proposition \ref{prop:uniform Harnack} gives that for $u,v\in \Lambda_n$,
\begin{multline}\label{eq:proof propreg2 1}
	\mathbb E_u'[G_\beta^{\Lambda}(X_\tau,x)]\leq C_{\rm RW}\mathbb E'_v[G_\beta^{\Lambda}(X_\tau,x)]+\frac{\eta}{8}\max\{G_\beta^{\Lambda}(w,x):w\in \Lambda_{(1+\alpha)n}\}\\+2C_{\rm RW}\max \Big\{|G_\beta^\Lambda(w,x)-G_\beta^{\Lambda}(w',x)|: w,w'\in \Lambda_{\lfloor \delta M\rfloor}(z), \: z\in \partial \Lambda_{n+M}\Big\} 
\end{multline}
where we used that for this choice of $M,m$ and for $n$ large enough one has: 
\begin{align}
3m(n/m)^{1/10}&\leq \frac{\delta \alpha}{12}n\leq \frac{\delta\alpha}{6}n-\delta-1\leq \lfloor\delta M\rfloor,\\
n+M+2m&\leq (1+\alpha)n.\end{align} 
Combining \eqref{eq:proof propreg2 1} with \cite[Corollary~A.6]{DumPan24Perco} (applied to $\frac{\eta}{16C_{\rm RW}}$ and $A\approx (\alpha\delta)^{-10/9}$) provides $T=T(\eta/8,\alpha,d)>0$ and $\varepsilon=\varepsilon(T)$ such that $(1+\varepsilon)^T\leq 2$ and for every $\varphi\in [1-\varepsilon,1+\varepsilon]$,
\begin{multline}\label{eq:proof propreg2 1.5}
	\mathbb E_u'[\varphi^{\tau\wedge T}G_\beta^{\Lambda}(X_{\tau\wedge T},x)]\leq C_{\rm RW}\mathbb E'_v[\varphi^{\tau\wedge T}G_\beta^{\Lambda}(X_{\tau\wedge T},x)]+\frac{\eta}{4}\max\{G_\beta^{\Lambda}(w,x):w\in \Lambda_{(1+\alpha)n}\}\\+2C_{\rm RW}\max \Big\{|G_\beta^\Lambda(w,x)-G_\beta^{\Lambda}(w',x)|: w,w'\in \Lambda_{\lfloor \delta M\rfloor}(z), \: z\in \partial \Lambda_{n+M}\Big\}. 
\end{multline}
Additionally, thanks to \cite[Proposition~A.5]{DumPan24Perco}, for every $u\in \Lambda_n$,
\begin{equation}\label{eq:proof propreg2 2}
	\mathbb E_u'[\tau]\leq 9d\left(\frac{n+M}{n}\right)^2\leq C_1(1+\alpha)^2\delta^{-20/9},
\end{equation}
for some $C_1=C_1(d)>0$.
%
We are now in a position to prove the desired result.

To the cost of diminishing it, we assume Proposition \ref{prop: theorem beta start = beta_c} holds with $\lambda_0$, and we additionally require that $K\lambda_0\le \varepsilon$ where $K$ is given by the same proposition.
 Also, fix $u,v\in \Lambda_n$. By definition of $m$ and \eqref{eq:take delta smaller}, we observe that $6m<n$. Since we also have $n\le 6L_\beta(\varepsilon)$, Proposition~\ref{prop:  theorem beta start = beta_c} gives
\begin{equation}
1-\varepsilon\le \varphi_\beta(\Lambda_{m-1})\le 1+K\lambda\le 1+\varepsilon.
\end{equation}
We now set $\varphi:=\varphi_\beta(\Lambda_{m-1})$. Iterating the two bounds of Lemma~\ref{Lem: Simon-Lieb WSAW} with $S$ being translates of $\Lambda_{m-1}$, until time $\tau\wedge T$, and using Lemma~\ref{lem:bound error} gives
\begin{align}
G_\beta^\Lambda(u,x)&\stackrel{\eqref{eq:SL}}\le  \mathbb E_{u}'[\varphi^{\tau\wedge T} G_\beta^\Lambda(X_{\tau\wedge T}',x)],\label{eq:hea}\\
G_\beta^\Lambda(v,x)&\stackrel{\eqref{eq:reversed SL}}\ge \mathbb E_{v}'[\varphi^{\tau\wedge T} G_\beta^\Lambda(X_{\tau\wedge T}',x)]-K\lambda \,\mathbb E_v'\Big[\sum_{s=0}^{\tau\wedge T}\varphi^s\Big]\max\{G_\beta^\Lambda(w,x):w\in \Lambda_{n+M+m}\}\notag\\
&\stackrel{\phantom{\eqref{eq:reversed SL}}}\ge \mathbb E_v'[\varphi^{\tau\wedge T} G_\beta^\Lambda(X_{\tau\wedge T}',x)]-2K\lambda \mathbb E_v'[\tau+1]\max\{G_\beta^\Lambda(w,x):w\in \Lambda_{n+M+m}\}\label{eq:hda},
\end{align}
where in the last inequality we used that $(1+\varepsilon)^T\leq 2$.
Combining \eqref{eq:hea} and \eqref{eq:proof propreg2 1.5} gives
\begin{multline}\label{eq:proof propreg2 3}
	G_\beta^\Lambda(u,x)\leq C_{\rm RW}\mathbb E'_v[\varphi^{\tau\wedge T}G_\beta^{\Lambda}(X_{\tau\wedge T},x)]+\frac{\eta}{4}\max\{G_\beta^{\Lambda}(w,x):w\in \Lambda_{(1+\alpha)n}\}\\+2C_{\rm RW}\max \Big\{|G_\beta^\Lambda(w,x)-G_\beta^{\Lambda}(w',x)|: w,w'\in \Lambda_{\lfloor \delta M\rfloor}(z), \: z\in \partial \Lambda_{n+M}\Big\}. 
\end{multline}
Assume that $\lfloor\delta M\rfloor\geq 6$, which occurs as soon as $n\geq 6\alpha^{-1}\delta^{-1}(L+1)$. By Corollary \ref{cor: improved reg parity-wise} (recall that it is applied to $\tfrac{\eta}{4C_{\rm RW}}$),
\begin{multline}\label{eq:proof propreg2 4}
	2C_{\rm RW}\max \Big\{|G_\beta^\Lambda(w,x)-G_\beta^{\Lambda}(w',x)|: w,w'\in \Lambda_{\lfloor \delta M\rfloor}(z), \: z\in \partial \Lambda_{n+M}\Big\}\\\leq \frac{\eta}{2}\max\{G_\beta^{\Lambda}(w,x): w \in \Lambda_{n+3M}\}.
\end{multline}
Moreover, if we decrease $\lambda_0$ so that $2K\lambda_0 (C_1(1+\alpha)^2\delta^{-20/9}+1)\leq \tfrac{\eta}{4C_{\rm RW}}$, \eqref{eq:proof propreg2 2} implies 
\begin{equation}\label{eq:proof propreg2 5}
	2K\lambda \mathbb E_v'[\tau+1]\leq \frac{\eta}{4 C_{\rm RW}}.
\end{equation}
Plugging \eqref{eq:hda} in \eqref{eq:proof propreg2 3}, and using \eqref{eq:proof propreg2 4} and \eqref{eq:proof propreg2 5} gives
\begin{equation}
	G_\beta^{\Lambda}(u,x)\leq C_{\rm RW}G_\beta^{\Lambda}(v,x)+\eta\max\{G_\beta^\Lambda(w,x):w\in \Lambda_{(1+\alpha)n}\},
\end{equation}
where we used that $n+3M\leq (1+\alpha)n$. Since $u$ and $v$ are arbitrary, this concludes the proof.
 \end{proof}

%

\subsection{Conclusion}

We are now in a position to prove Theorem \ref{thm:main2wsaw}. We will first prove the bounds for $x$ such that $|x|\lesssim L_\beta(\varepsilon)$  for a well-chosen $\varepsilon$. The idea will be to use the Harnack-type estimate derived in Proposition \ref{prop:regularity2} to turn the averaged lower bound of \eqref{eq:averaged lower bound half-space} into a point-wise lower bound on the half-space two-point function. The full-space lower bound below distance $L_\beta(\varepsilon)$ will follow from the half-space estimate by an application of Lemma \ref{Lem: Simon-Lieb WSAW}. Finally, we will obtain the result for large values of $|x|$ by induction.

We start by lower bounding the half-space two-point function at scale below $6L_\beta(\varepsilon)$ (for some technical reasons we will need a multiple of $L_\beta(\varepsilon)$ later).

\begin{Lem}\label{lem:lower below half-space}
Fix $d>4$. There exist $c,\varepsilon_0,\lambda_0>0$ such that for every $\lambda<\lambda_0$, every $\varepsilon<\varepsilon_0$, every $\tfrac1{2d}\le \beta\le \beta_c$, and every $x\in \mathbb H$ with $x_1=|x|\le 6L_\beta(\varepsilon)$,
\begin{equation}\label{eq:op}
G_\beta^{\mathbb H}(0,x)\ge \frac{c}{(1\vee|x|)^{d-1}}.
\end{equation}
\end{Lem}
\begin{proof} Let $\varepsilon_0\in(0,\tfrac12)$ to be fixed later. Recall from \eqref{eq: def An} that
\begin{equation}
A_n=\{x\in \mathbb Z^d:x_1=|x|=n\}.
\end{equation}
Thanks to \eqref{eq:averaged lower bound half-space}, there exists $c_0=c_0(d)>0$ such that for every $n< L_\beta(\varepsilon)$,
\begin{align}\label{eq:lower average}
\frac1{|A_n|}\sum_{y\in A_n}G_{\beta}^{\mathbb H}(0,y)&\geq\frac{c_0}{n^{d-1}}.
\end{align}

We now want to turn this averaged bound into a point-wise one. Let $\eta>0$ to be chosen sufficiently small. Let $\lambda_0,\varepsilon_0,N_0$ be given by Proposition \ref{prop:regularity2} applied to $\alpha=\tfrac{1}{12}$ and $\eta$. Let $\lambda<\lambda_0$ and $\varepsilon<\varepsilon_0$. Consider  $x\in A_N$ with $N\le6L_\beta(\varepsilon)$ and set $n:=\lfloor N/6\rfloor-1\le L_\beta(\varepsilon)-1$. Assume first that $N\geq N_0$, i.e.\ $n\geq 6N_0$.
Proposition~\ref{prop:regularity2} (as illustrated in Figure \ref{fig:howtofindalpha}) implies that for every $y\in A_n$, 
\begin{align}\label{eq:h1}
C_{\rm RW}G_{\beta}^{\mathbb H}(0,x)\ge G_{\beta}^{\mathbb H}(0,y)-\eta \max\{G_{\beta}^{\mathbb H}(0,z):z_1\ge n/2\}.
\end{align}
Averaging over $y\in A_n$, using the lower bound \eqref{eq:lower average} and the upper bound \eqref{eq:bound full plane} gives
\begin{align}\label{eq:h1}
C_{\rm RW}G_{\beta}^{\mathbb H}(0,x)\ge \frac{c_0}{n^{d-1}}-\eta \cdot \frac{C}{(n/2)^{d-1}}.
\end{align}
Choosing $\eta=\eta(c_0,C,d)$ such that $\eta C2^{d-1}\leq \tfrac{c_0}{2}$ (which affects how small $\lambda_0,\varepsilon_0$ have to be and how large $N_0$ is) implies the lower bound when $|x|\geq N_0$. The proof follows readily by choosing $c_1$ small enough such that
\begin{equation}
	\min\{G_{1/(2d)}^{\mathbb H}(0,x):x_1=|x|\leq N_0\}\geq \frac{c_1}{N_0^{d-1}},
\end{equation}
and setting $c:=c_0\wedge c_1$.
\end{proof}

\begin{figure}
	\begin{center}
		\includegraphics{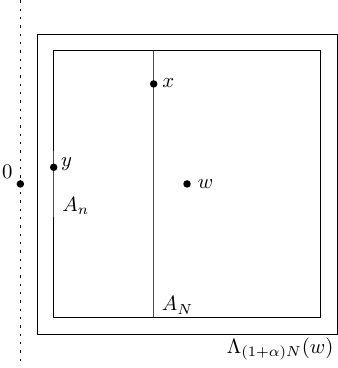}
		\put(-175,170){$\partial\mathbb H$}
		\caption{An illustration of the application of Proposition \ref{prop:regularity2} in the proof of Lemma \ref{lem:lower below half-space}. We wish to argue that $G_\beta^{\mathbb H}(0,y)$ and $G_\beta^{\mathbb H}(0,x)$ are of the same order. The red segments represent the sets $A_n$ and $A_N$. The boxes are centered at $w=\tfrac{7N}{6}\mathbf{e}_1$. The parameter $\alpha$ is chosen so that $\Lambda_{(1+\alpha)N}(w)\subset \{z\in \mathbb H: \: z_1\geq n/2\}$, i.e.\ $\alpha=\tfrac{1}{12}$.}
		\label{fig:howtofindalpha}
	\end{center}
\end{figure}

We now turn to the full-plane lower bound below scale $5L_\beta(\varepsilon)$.
\begin{Lem}\label{lem:lower below full-space}
Fix $d>4$. There exist $c,\varepsilon_0,\lambda_0>0$ such that for every $\lambda<\lambda_0$, every $\varepsilon<\varepsilon_0$, every $\tfrac1{2d}\le \beta\le \beta_c$, and every $|x|\le5L_\beta(\varepsilon)$,
\begin{equation}\label{eq:lb full below 5L}
G_\beta(0,x)\ge \frac{c}{(1\vee|x|)^{d-2}}.
\end{equation}
\end{Lem}

\begin{proof}Let $\lambda_0, \varepsilon_0$ be given by Lemma \ref{lem:lower below half-space}. We will choose $\lambda_0$ even smaller below. Let $\lambda<\lambda_0$ and $\varepsilon<\varepsilon_0$. By symmetry, we may consider $x\in A_n$, where $n=|x|$. Lemma~\ref{Lem: Simon-Lieb WSAW} applied to $S=\mathbb H_k$ and $\Lambda=\mathbb H_{k+1}$ gives that
\begin{multline}\label{eq: lower bound typical full space}
G_\beta(0,x)=G_\beta^{\mathbb H}(0,x)+\sum_{k\geq 0}\Big(G_\beta^{\mathbb{H}_{k+1}}(0,x)-G_\beta^{\mathbb H_k}(0,x)\Big)
\\\ge \sum_{k\ge 0}\sum_{\substack{y\in \mathbb H_k,\\ z\notin \mathbb H_k,\:y\sim z}}G_\beta^{\mathbb H_k}(0,y)\beta G_\beta^{\mathbb H_{k+1}}(z,x)-\lambda \sum_{k\geq 0}\sum_{u\in \mathbb H_k}E_\beta^{\mathbb H_k,\mathbb H_{k+1}}(u)G_\beta^{\mathbb H_{k+1}}(u,x).
\end{multline}
Looking at the first sum, we see that
\begin{align*}
 	\sum_{k\geq 0}\sum_{\substack{y\in \mathbb H_k,\\ z\notin \mathbb H_k,\:y\sim z}}G_\beta^{\mathbb H_k}(0,y)\beta G_\beta^{\mathbb H_{k+1}}(z,x)&\geq \sum_{k=0}^{n\wedge (L_\beta(\varepsilon)-1)}\frac1{2d}\varphi_\beta(\Lambda_k)\min\{G_\beta^{\mathbb H_{k+1}}(z,x):-z_1=|z|=k+1\}	
 	\\&\geq \sum_{k=0}^{n\wedge (L_\beta(\varepsilon)-1)}\frac1{2d}(1-\varepsilon)\frac{c}{(n+k)^{d-1}}
 	\\&\geq \frac{c_1}{n^{d-2}},
\end{align*}
where $c_1=c_1(d)>0$, and where we restricted our attention to special positions of $y$ and used both the bound $\varphi_\beta(\Lambda_k)\ge 1-\varepsilon$ provided by $k\le n \wedge (L_\beta(\varepsilon)-1)$ and the lower bound  from Lemma~\ref{lem:lower below half-space} (note that $n+n \wedge L_\beta(\varepsilon)\le 6L_\beta(\varepsilon)$ for $n\le 5L_\beta(\varepsilon)$).

Turning to the ``error'' term in \eqref{eq: lower bound typical full space}, we take the position of $u$ into account and use the upper bound \eqref{eq:bound half plane} to get 
\begin{align}
		\sum_{k\geq 0}\sum_{p\geq 0}\sum_{u\in \partial \mathbb H_{k-p}}\sum_{\substack{y\in \mathbb H_k\\z\notin \mathbb H_k\\y\sim z}} &G_\beta^{\mathbb H_k}(0,u)G_\beta^{\mathbb H_k}(u,y)\beta G_\beta^{\mathbb H_{k+1}}(z,u)G_\beta^{\mathbb H_{k+1}}(u,x)\nonumber
		\\&\leq \sum_{k\geq 0}\sum_{p\geq 0}\frac{C}{(p+1)^{d-1}}\cdot \varphi_\beta(\mathbb H_p)\cdot \sum_{u\in \partial \mathbb H_{k-p}}G_\beta^{\mathbb H_k}(0,u)G_\beta^{\mathbb H_{k+1}}(u,x).\label{eq: proof full space lower 1}
		\end{align}
Using Proposition \ref{prop: theorem beta start = beta_c}, we have that $\varphi_\beta(\mathbb H_p)\leq 1+K\lambda$. We bound \eqref{eq: proof full space lower 1} differently according to the value of $p$. 

First, use \eqref{eq:bound full plane} to get
\begin{align}
	\sum_{p\geq (n+k)/2}\frac{C(1+K\lambda)}{(p+1)^{d-1}}\sum_{u\in \partial \mathbb H_{k-p}}G_\beta^{\mathbb H_k}(0,u)G_\beta^{\mathbb H_{k+1}}(u,x)&\leq \frac{C_1}{(n+k)^{d-1}}\sum_{u\in \mathbb Z^d}G_{\beta}(0,u)G_{\beta}(u,x)
	\\&\leq \frac{C_2}{(n+k)^{d-1}},\label{eq: proof full space lower 2}
\end{align} where $C_1,C_2>0$ only depend on $C$ and $d$.
%

Second, turning to the contribution coming from $p< \tfrac{n+k}{2}$, we write
\begin{equation}
	\max_{u\in \partial \mathbb H_{k-p}}G_\beta^{\mathbb H_{k+1}}(u,x)\stackrel{\eqref{eq: half plane at distance $k$ from half space}}\leq \frac{2^dC(p+2)}{(n+k)^{d-1}}, \qquad \sum_{u\in \partial \mathbb H_{k-p}}G_\beta^{\mathbb H_{k}}(0,u)\stackrel{\eqref{eq: sum at distance k from halfplane}}\leq \frac{4(p+1)}{\beta},
\end{equation}
where we used that for $u\in \partial \mathbb H_{k-p}$, one has $|u_1-x_1|=n+k-p\geq \tfrac{n+k}{2}$. Hence,
\begin{align}
	\sum_{p< (n+k)/2}\frac{C(1+K\lambda)}{(p+1)^{d-1}}\sum_{u\in \partial \mathbb H_{k-p}}G_\beta^{\mathbb H_k}(0,u)G_\beta^{\mathbb H_{k+1}}(u,x)&\leq C_3\sum_{p\leq\tfrac{n+k}{2}}\frac{1}{(p+1)^{d-3}}\frac{1}{(n+k)^{d-1}}
	\\&\leq \frac{C_4}{(n+k)^{d-1}},\label{eq: proof full space lower 3}
\end{align}
where $C_3,C_4>0$ only depend on $C$ and $d$. Combining \eqref{eq: proof full space lower 2} and \eqref{eq: proof full space lower 3} gives that
\begin{align}
	\sum_{k\geq 0}\sum_{p\geq 0}\sum_{u\in \partial \mathbb H_{k-p}}\sum_{\substack{y\in \mathbb H_k\\z\notin \mathbb H_k\\y\sim z}} G_\beta^{\mathbb H_k}(0,u)G_\beta^{\mathbb H_k}(u,y)\beta G_\beta^{\mathbb H_{k+1}}(z,u)G_\beta^{\mathbb H_{k+1}}(u,x)&\leq \sum_{k\geq 0}\frac{C_5}{(n+k)^{d-1}}\nonumber\\
	&\leq \frac{C_6}{n^{d-2}},
\end{align}
where $C_5,C_6$ only depend on $C$ and $d$. The proof follows by choosing $\lambda_0$ small enough so that $\lambda_0C_6<c_1/2$, and setting $c=c_1/2$.
%
%
\end{proof}
We are now in a position to prove Theorem~\ref{thm:main2wsaw}.
\begin{proof}[Proof of Theorem~\textup{\ref{thm:main2wsaw}}] Let $\lambda_0,\varepsilon_0>0$ be such that the previous two lemmata apply. We will (potentially) choose them even smaller below. Let $\lambda<\lambda_0$ and $\varepsilon<\varepsilon_0$.

Set $L'_\beta:=L_\beta(\varepsilon)$. If $\beta=\beta_c$ then $L_\beta=L_\beta'=\infty$ and the Lemmata \ref{lem:lower below half-space} and \ref{lem:lower below full-space} are sufficient to conclude. We therefore assume $\beta<\beta_c$. By Lemma \ref{lem: comparison of diff l} it suffices to prove the lower bounds with $L'_\beta$ instead of $L_\beta$ in the exponential.

We already have the corresponding lower bounds for $|x|\le 5L'_\beta$. Let us turn to the general case. We start with the full-space estimate and then explain how to adapt the argument to derive the half-space estimate.
\paragraph{The full-space case.} Let $c>0$ be given by Lemma \ref{lem:lower below full-space}. Introduce, for $k\ge0$, 
\begin{align}
m_k&:=\min\{G_\beta(0,x):x\in \Lambda_{(k+1)L_\beta'-1}\},
\end{align}
and 
\begin{equation}
S_k:=\{x\in\mathbb Z^d:kL'_\beta\le |x|<(k+1)L'_\beta\}.
\end{equation}
We prove by induction for $k\ge 4$ that for some $c_1>0$,
\begin{equation}
m_k\ge \frac{c}{(5L'_\beta)^{d-2}}c_1^{k-4}.
\end{equation}
For $k=4$, it is simply \eqref{eq:lb full below 5L}. We now assume that $k\ge 5$.

\begin{Claim}\label{claim:last proof} There exists $c_0=c_0(d)>0$ such that, for every $x\in S_k$,
\begin{equation}\label{eq:lb1}
	G_\beta(0,x)\geq 2c_0m_{k-1}-2K\lambda M_2(x)
\end{equation}
where
for $\ell\geq 1$,
\begin{align}
M_\ell(x)&:=\max\{G_\beta(0,y):y\in \Lambda_{\ell L'_\beta}(x)\}.\end{align}
\end{Claim}
\begin{proof}[Proof of Claim \textup{\ref{claim:last proof}}]
Without loss of generality we can assume that $x_i\geq 0$ for all $1\leq i\leq d$. Let $\Lambda:=\Lambda_{L_\beta'-1}$ and 
\begin{equation}
H_1:=\partial \Lambda \cap \{u \in \mathbb Z^d: u_2=-(L_\beta'-1),\: u_1\leq 0, \:u_i\leq 0, \: \forall 3\leq i\leq d\}.\end{equation} Apply a first time \eqref{eq:reversed SL} and Lemma \ref{lem:bound error} to get
\begin{equation}
	G_\beta(0,x)\geq \sum_{\substack{y\in (\Lambda\cap H_1+x)\\z \notin (\Lambda+x)\\y\sim z}}G_\beta^{\Lambda+x}(x,y)\beta G_\beta(z,0)-K\lambda M_1(x).
\end{equation}
Then, letting 
\begin{equation}H_2:=\partial \Lambda \cap \{u\in \mathbb Z^d: u_1=-(L_\beta'-1), \: u_i\leq 0, \:\forall 2\leq i \leq d\},\end{equation} a new application of \eqref{eq:reversed SL} and Lemma \ref{lem:bound error} gives,
\begin{align}
	G_\beta(0,x)\geq \sum_{\substack{y\in (\Lambda\cap H_1+x)\\z \notin (\Lambda+x)\\z\sim y}}&G_\beta^{\Lambda+x}(x,y)\beta\sum_{\substack{u\in (\Lambda\cap H_2+z)\\v \notin (\Lambda+z)\\u\sim v}}G^{\Lambda+z}_\beta(z,u)\beta G_\beta(v,0)
	-K\lambda M_2(x)-K\lambda M_1(x).
\end{align}
Now, by construction $v$ as above must lie in $\Lambda_{kL_\beta'-1}$ (see Figure \ref{fig:doublereverseSL}) and therefore, by symmetry and by the definition of $L_\beta'$,
\begin{equation}
	G_\beta(0,x)\geq \Big(\frac{1}{2d}\frac{1}{2^{d-1}}(1-\varepsilon)\Big)^2m_{k-1}-2K\lambda M_2(x),
\end{equation}
and \eqref{eq:lb1} follows from setting $2c_0:=(\tfrac{1}{2d}\tfrac{1}{2^{d-1}}(1-\varepsilon))^2$.
\end{proof}

%
%
\begin{figure}[htb]
	\begin{center}
		\includegraphics{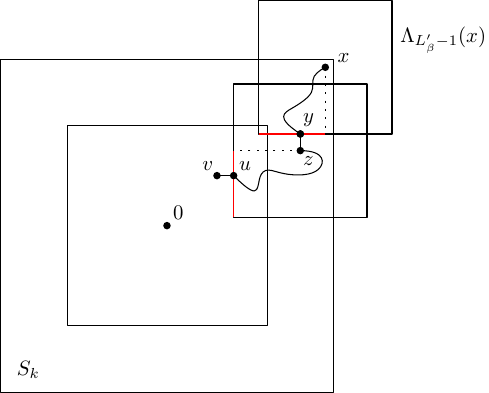}
		\caption{An illustration of the proof of \eqref{eq:lb1}. The situation depicted here is somehow the \emph{least} favorable. The sets $H_1$ and $H_2$ are the red bold lines. With the two successive applications of Lemma \ref{Lem: Simon-Lieb WSAW}, we ``drive'' the point $x$ towards the box $\Lambda_{kL_\beta'-1}$. This provides a recurrence relation between $m_k$ and $m_{k-1}$.}
		\label{fig:doublereverseSL}
	\end{center}
\end{figure}
If $x\in S_k$ is such that $2K\lambda M_2(x)\le G_\beta(0,x)$, then \eqref{eq:lb1} gives
\begin{equation}\label{eq:lb2}
G_\beta(0,x)\ge c_0 m_{k-1}.
\end{equation}
Moreover, when $x\in \Lambda_{kL_\beta'-1}$,
\begin{equation}
	G_\beta(0,x)\geq m_{k-1}\geq c_0m_{k-1}.
\end{equation}
As a consequence, we find $m_k\ge c_0m_{k-1}$ and therefore advance the induction hypothesis (with $c_1=c_0$), except if there is $x\in S_k$ such that $2K\lambda M_2(x)> G_\beta(0,x)$. We show below that this is in fact impossible by proceeding by contradiction. 

Let $\eta<\min\{1/(4C_{\rm RW}),c/(4C5^{d-2})\}$ and $\lambda_0<\eta/K$. Also, (potentially) decrease $\varepsilon_0$ so that Proposition~\ref{prop:regularity2} holds true for this $\eta$ and for $\alpha=1$. Let $N_0=N_0(\eta,1,d)$ be given by the same proposition. We may additionally assume that $2L'_\beta\geq N_0$ by choosing $\lambda_0$ even smaller. Indeed, observe that, following a similar argument as in \eqref{eq:phi beta rw =1}, one has that $L'_{(2d)^{-1},\lambda=0}=\infty$. Since, by Remark \ref{rem: beta_c}, one has $\tfrac{1}{2d}\leq \beta\leq \beta_c(\lambda)\leq \tfrac{1}{2d}+O(\lambda)$, it is possible (by continuity) to choose $\lambda$ small enough so that $2L'_\beta\geq N_0$. 

Let $\ell$ be such that $0\notin \Lambda_{(\ell+4)L'_\beta}(x)$. By Proposition \ref{prop:regularity2}, if $y\in \Lambda_{\ell L'_\beta}(x)$,
\begin{align}\notag
	\max\{&G_\beta(0,w): w\in \Lambda_{2L'_\beta}(y)\}\\&\leq C_{\rm RW}\min\{G_\beta(0,w):w\in \Lambda_{2L'\beta}(y)\cap \Lambda_{\ell L_\beta'}(x)\}+\eta \max \{G_\beta(0,w):w\in \Lambda_{4L_\beta'}(y)\} \notag\\&\leq C_{\rm RW}M_\ell(x)+\eta M_{\ell+4}(x).\label{eq:recursive equation M}
\end{align}
Optimising the above displayed equation over $y\in \Lambda_{\ell L'_\beta}(x)$ gives, for such $\ell$,
\begin{equation}\label{eq: renormalisation equation}
M_{\ell+2}(x)\le C_{\rm RW}M_\ell(x)+\eta M_{\ell+4}(x).
\end{equation}

\begin{Claim}\label{claim: recursion} Under our choices of $\eta$ and $\lambda_0$ and the assumption $M_0(x)=G_\beta(0,x)<2K\lambda M_2(x)$, one has, for every $k\geq 0$ such that $0\notin \Lambda_{(2k+2)L'_\beta}(x)$,
\begin{equation}\label{eq:recursive bound}
M_{2k-1}(x)\le M_{2k}(x)\le 2\eta M_{2k+2}(x).\end{equation}
\end{Claim}
\begin{proof}[Proof of Claim \textup{\ref{claim: recursion}}] We proceed by induction. Since $\lambda<\eta/K$, the assumption gives $M_0(x)\leq 2\eta M_2(x)$ and the result for $k=0$. Assume that $k\geq 1$ is such that $0\notin \Lambda_{(2k+2)L'_\beta}(x)$ and that \eqref{eq:recursive bound} holds for $k-1$. Then, applying \eqref{eq: renormalisation equation} for $\ell=2k-2$ gives
\begin{equation}
	M_{2k}(x)\leq C_{\rm RW}M_{2k-2}(x)+\eta M_{2k+2}(x).
\end{equation}
By the induction hypothesis, $M_{2k-2}(x)\leq 2\eta M_{2k}(x)$. Plugging this inequality in the above equation gives
\begin{equation}
	M_{2k}(x)\leq \frac{\eta}{1-2\eta C_{\rm RW}}M_{2k+2}(x)\leq 2\eta M_{2k+2}(x),
\end{equation}
where we used the assumption that $\eta\leq 1/(4C_{\rm RW})$. Using the trivial bound $M_{2k-1}(x)\leq M_{2k}(x)$ concludes the proof.
\end{proof}
Applying the above claim to $2k=L:=2\lfloor \tfrac{1}{2}(|x|/L'_\beta-3)\rfloor$, we obtain that 
\begin{equation}
\frac{c}{(5L'_\beta)^{d-2}}\stackrel{\eqref{eq:lb full below 5L}}\le m_4\le M_{L}(x)\le 2\eta M_{L+2}(x)\stackrel{\eqref{eq:upper bound below L}}\le 2\eta \frac{2C}{(L'_\beta)^{d-2}}.\end{equation}
The choice of $\eta$  leads to a contradiction, therefore concluding the proof in the case of the full-space.

\paragraph{The half-space case.} The proof follows the exact same strategy as the full-space case. We make a small change and now choose $\eta<\min\{1/(4C_{\rm RW}),c/(4C6^{d-1})\}$ and $\lambda_0<\eta/K$. We let $c>0$ be given by Lemma \ref{lem:lower below half-space}. We define, for $k\geq 0$,
\begin{equation}
	\tilde m_k:=\min \{G_\beta^{\mathbb H}(0,x): x_1=|x|\leq (k+1)L_\beta'-1\},
\end{equation}
and
\begin{equation}
	\tilde{S}_k:=\{x\in \mathbb H: x_1=|x|, \: kL_\beta'\leq |x|<(k+1)L_\beta'\}.
\end{equation}
We prove by induction that there exists $c_2>0$ such that for every $k\geq 5$,
\begin{equation}
	\tilde{m}_k\geq \frac{c}{(6L_\beta')^{d-1}}c_2^{k-5}.
\end{equation}
For $k=5$, this follows from \eqref{eq:op}. We thus assume that $k\geq 6$.

\begin{Claim}\label{claim: first claim half case}
There exists $\tilde{c}_0=\tilde{c}_0(d)>0$ such that, for every $x\in \tilde{S}_k$,
\begin{equation}\label{eq:lb1 half}
 G_\beta^{\mathbb H}(0,x)\geq 2\tilde{c}_0\tilde{m}_{k-1}-2K\lambda \tilde{M}_2(x).
\end{equation}
where
for $\ell\geq 1$,
\begin{align}
\tilde{M}_\ell(x)&:=\max\{G_\beta^{\mathbb H}(0,y):y\in \Lambda_{\ell L'_\beta}(x)\}.\end{align}
\end{Claim}
\begin{proof}[Proof of Claim \textup{\ref{claim: first claim half case}}] The proof is exactly the same as the one of Claim~\ref{claim:last proof} except that in the applications of Lemma \ref{Lem: Simon-Lieb WSAW} we choose $\Lambda=\mathbb H$. In particular, we may set $\tilde{c}_0:=c_0$.
\end{proof}
If $x\in \tilde{S}_k$ is such that $2K\lambda \tilde{M}_2(x)\le G_\beta^{\mathbb H}(0,x)$, then \eqref{eq:lb1 half} gives
\begin{equation}\label{eq:lb2 half}
G_\beta^{\mathbb H}(0,x)\ge \tilde{c}_0 \tilde{m}_{k-1}.
\end{equation}
Moreover, when $x_1=|x|\leq kL_\beta'-1$,
\begin{equation}
	G_\beta^{\mathbb H}(0,x)\geq \tilde{m}_{k-1}\geq \tilde{c}_0\tilde{m}_{k-1}.
\end{equation}
We thus find that $\tilde{m}_k\ge \tilde{c}_0\tilde{m}_{k-1}$, and therefore advance the induction hypothesis (with $c_2=\tilde{c}_0$), except if there is $x\in \tilde{S}_k$ such that $2K\lambda \tilde{M}_2(x)> G_\beta^{\mathbb H}(0,x)$. Again, we show that this fact is impossible by contradiction.

Using the exact same strategy as in \eqref{eq: renormalisation equation}, we obtain that for $\ell$ such that $0\notin \Lambda_{(\ell+4)L_\beta'}(x)$, 
\begin{equation}
	\tilde{M}_{\ell+2}\leq C_{\rm RW}\tilde{M}_\ell(x)+\eta \tilde{M}_{\ell+4}(x),
\end{equation}
with $\tilde M_0(x)=G_\beta^{\mathbb H}(0,x)<2K\lambda \tilde{M}_2(x)$. The following claim is derived in the exact same way as Claim \ref{claim: recursion}.
\begin{Claim}\label{claim: recursion half} Under our choices of $\eta$ and $\lambda_0$ and the assumption $\tilde{M}_0(x)=G_\beta^{\mathbb H}(0,x)<2K\lambda \tilde{M}_2(x)$, one has, for every $k\geq 0$ such that $0\notin \Lambda_{(2k+2)L'_\beta}(x)$,
\begin{equation}\label{eq:recursive bound}
\tilde{M}_{2k-1}(x)\le \tilde{M}_{2k}(x)\le 2\eta \tilde{M}_{2k+2}(x).\end{equation}
\end{Claim}
Applying the above claim to $2k=L:=2\lfloor \tfrac{1}{2}(|x|/L'_\beta-3)\rfloor$, we obtain that 
\begin{equation}
\frac{c}{(6L'_\beta)^{d-1}}\stackrel{\eqref{eq:op}}\le \tilde{m}_5\le \tilde{M}_{L}(x)\le 2\eta \tilde{M}_{L+2}(x)\stackrel{\eqref{eq:upper half-spacebound below L}}\le 2\eta \frac{2C}{(L'_\beta)^{d-1}}.\end{equation}
Again, the choice of $\eta$  leads to a contradiction. This concludes the proof in the case of the half-space.

\end{proof}

\appendix
\section{Appendix: proof of Lemma~\ref{Lem: Simon-Lieb WSAW}}\label{appendix:sl}

In this section, we prove Lemma \ref{Lem: Simon-Lieb WSAW}. Let $d\geq 2$ and $0<\beta<\beta_c$. Fix $S,\Lambda$ two subsets of $\mathbb Z^d$ satisfying $0\in S\subset \Lambda$. Finally, fix $x\in \Lambda$. Our goal is to show that
\begin{align}\label{eq:SLproof}
G_\beta^\Lambda(0,x)&\le G_\beta^S(0,x)+ \sum_{\substack{y\in S\\ z\in \Lambda\setminus S\\ y\sim z}}G_\beta^S(0,y)\beta G_\beta^\Lambda(z,x),\\
G_\beta^\Lambda(0,x)&\ge G_\beta^S(0,x)+\sum_{\substack{y\in S\\ z\in \Lambda\setminus S\\ y\sim z}}G_\beta^S(0,y)\beta G_\beta^\Lambda(z,x)- \lambda \sum_{u\in S}E_\beta^{S,\Lambda}(u)G_\beta^\Lambda(u,x),\label{eq:reversed SLproof}
\end{align}
where 
\begin{equation}
E_\beta^{S,\Lambda}(u)=\sum_{\substack{y \in S\\ z\in \Lambda\setminus S\\ y\sim z}}G_\beta^S(0,u)G_\beta^S(u,y) \beta G_\beta^\Lambda(z,u).
\end{equation}

Observe that
\begin{equation}\label{eq:proofSL1}
	G_\beta^{\Lambda}(0,x)-G_\beta^S(0,x)=\sum_{\substack{\gamma: 0 \rightarrow x\subset \Lambda\\ \gamma \cap S^c\neq \emptyset}}\beta^{|\gamma|}\rho(\gamma)=\sum_{\substack{y\in S\\ z\in \Lambda\setminus S\\ y\sim z}}\sum_{\substack{\gamma_1:0\rightarrow y\subset S\\ \gamma_2:z\rightarrow x\subset \Lambda}}\beta^{|\gamma_1|+|\gamma_2|+1}\rho(\gamma_1\circ(yz)\circ \gamma_2),
\end{equation}
where\footnote{The notation $\gamma_1:0\rightarrow y\subset S$ means that the walk $\gamma_1$ goes from 0 to $y$ by remaining in $S$, and similarly for comparable subscripts.}we used that any $\gamma$ contributing to the middle sum in \eqref{eq:proofSL1} can be written as the concatenation $\gamma_1\circ (yz)\circ \gamma_2$, with $(yz)$ the first edge of $\gamma$ exiting $S$. The structure of the weights implies that 
\begin{equation}\label{eq:a}
    \rho(\gamma_1)\rho(\gamma_2)- \lambda\rho(\gamma_1)\rho(\gamma_2) \sum_{\substack{0\leq i \leq |\gamma_1|\\1\leq j \leq |\gamma_2|}} \mathds{1}_{\gamma_1(i)=\gamma_2(j)} \le \rho(\gamma_1\circ(yz)\circ\gamma_2)\leq  \rho(\gamma_1)\rho(\gamma_2),
\end{equation}
where, in the first inequality, we used that 
\begin{equation}
   1-\lambda  \sum_{\substack{0\leq i \leq |\gamma_1|\\1\leq j \leq |\gamma_2|}} \mathds{1}_{\gamma_1(i)=\gamma_2(j)} \le  \prod_{\substack{0\leq i \leq |\gamma_1|\\1\leq j \leq |\gamma_2|}} (1-\lambda\mathds{1}_{\gamma_1(i)=\gamma_2(j)}).\end{equation}
On the one hand, we have that
\begin{equation}\label{eq:proofSL2}
	\sum_{\substack{y\in S\\ z\in \Lambda\setminus S\\ y\sim z}}\sum_{\substack{\gamma_1:0\rightarrow y\subset S\\ \gamma_2:z\rightarrow x\subset \Lambda}}\beta^{|\gamma_1|+|\gamma_2|+1}\rho(\gamma_1)\rho(\gamma_2)=\sum_{\substack{y\in S\\ z\in \Lambda\setminus S\\ y\sim z}}G_\beta^S(0,y)\beta G_\beta^\Lambda(z,x),
\end{equation}
which gives \eqref{eq:SLproof} when we plug the second inequality of \eqref{eq:a} in \eqref{eq:proofSL1}.

On the other hand, if we decompose according to $(|\gamma_1|,|\gamma_2|)=(n,m)$ and the common value $u$ of $\gamma_1(i)$ and $\gamma_2(j)$ in the sum on the left-hand side of \eqref{eq:a}, we obtain that
\begin{multline}\label{eq:proofSL3}
	\sum_{\substack{y\in S\\ z\in \Lambda\setminus S\\ y\sim z}}\sum_{\substack{\gamma_1:0\rightarrow y\subset S\\\gamma_2:z\rightarrow x\subset \Lambda}}\beta^{|\gamma_1|+|\gamma_2|+1}\lambda \sum_{\substack{0\leq i \leq |\gamma_1|\\1\leq j \leq |\gamma_2|}}\mathds{1}_{\gamma_1(i)=\gamma_2(j)}
	\\=\lambda \sum_{u\in S}\sum_{\substack{y\in S\\ z\in \Lambda \setminus S\\ y\sim z}}\sum_{n,m\geq 0}\beta^{n+m+1}\sum_{\substack{0\leq i \leq n\\1\leq j \leq m}}\sum_{\substack{\gamma_1:0\rightarrow y \subset S\\ \gamma_2: z \rightarrow x \subset \Lambda}}\rho(\gamma_1)\rho(\gamma_2)\mathds{1}_{\gamma_1(i)=u, \: |\gamma_1|=n}\mathds{1}_{\gamma_2(j)=u, \: |\gamma_2|=m}.
\end{multline}
We further split $\gamma_1$ into $\gamma_1':0 \rightarrow u$ and $\gamma_1'':u\rightarrow y$, and do the same for $\gamma_2$. Note that by definition, $|\gamma_1'|=i$ and $|\gamma_1''|=n-i$. Similarly, $|\gamma_2'|=j$ and $|\gamma_2''|=m-j$. The Cauchy product formula implies that
\begin{equation}
	\sum_{n\geq 0}\sum_{i=0}^n \beta^{n+1}\sum_{\substack{\gamma_1':0\rightarrow u\subset S\\ \gamma_1'':u \rightarrow y\subset S}}\rho(\gamma_1')\mathds{1}_{|\gamma_1'|=i}\rho(\gamma_1'')\mathds{1}_{|\gamma_1''|=n-i}=G^S_\beta(0,u)G^S_\beta(u,y)\beta.
\end{equation}
Similarly,
\begin{equation}
	\sum_{m\geq 0}\sum_{j=1}^m \beta^{m}\sum_{\substack{\gamma_2':z\rightarrow u\subset \Lambda\\ \gamma_2'':u \rightarrow x\subset \Lambda}}\rho(\gamma_2')\mathds{1}_{|\gamma_2'|=j}\rho(\gamma_2'')\mathds{1}_{|\gamma_2''|=m-j}\leq G^\Lambda_\beta(z,u)G^\Lambda_\beta(u,x).
\end{equation}
Using \eqref{eq:a} to obtain that $\rho(\gamma_i)\leq \rho(\gamma_i')\rho(\gamma_i'')$ (for $i=1,2$), and plugging the two previously displayed equations in \eqref{eq:proofSL3} gives
\begin{equation}\label{eq:proofSL4}
	\sum_{\substack{y\in S\\ z\in \Lambda\setminus S\\ y\sim z}}\sum_{\substack{\gamma_1:0\rightarrow y\subset S\\\gamma_2:z\rightarrow x\subset \Lambda}}\beta^{|\gamma_1|+|\gamma_2|+1}\lambda \sum_{\substack{0\leq i \leq |\gamma_1|\\1\leq j \leq |\gamma_2|}}\mathds{1}_{\gamma_1(i)=\gamma_2(j)}\leq \lambda \sum_{u\in S}E_\beta^{S,\Lambda}(u)G_\beta^\Lambda(u,x).
\end{equation}
Combining the first inequality of \eqref{eq:a}, \eqref{eq:proofSL2}, and \eqref{eq:proofSL4} in \eqref{eq:proofSL1} gives \eqref{eq:reversed SLproof}. This concludes the proof.

\section{Appendix: random walk estimates}\label{appendix:rw estimates}

Let $u \in \mathbb Z^d$. Let $\mathbb P_u$ be the law of the simple random walk started at $u$. Also, define $\tau^n:=\inf\{k\geq 1: X_k\notin \mathbb H_n\}$.
\begin{Prop}\label{prop: halfspace green function estimate} Let $d>2$. There exists $\bfC_{\rm RW}>0$ such that, for every $x\in \mathbb H$,
\begin{equation}\label{eq:rw estimate 1}
	\mathbb E_0\Big[\sum_{\ell <\tau^0}\mathds{1}_{X_\ell=x}\Big]\leq \frac{\bfC_{\rm RW}}{(1\vee|x_1|)^{d-1}}.
\end{equation}
\end{Prop}
\begin{proof} This classical estimate follows from a combination of \cite[Proposition~8.1.1]{LawlerLimicRandomWalks2010} and \cite[Lemma~6.3.3]{LawlerLimicRandomWalks2010}.
\end{proof}
If $L\geq 1$, we let $\tau_L:=\inf\{k\geq 1:X_k\notin \Lambda_{L-1}\}$.
\begin{Prop}\label{prop: estimates srw} Let $d\geq 1$ and $\ell\geq 1$. Then, for every $\eta>0$, there exists $L_0=L_0(\ell,\eta,d)\geq \ell$ such that, for every $L\geq L_0$, for every $f:\Lambda_L\rightarrow \mathbb R^+$, and every $u,v \in \Lambda_\ell$, 
		\begin{equation}
		\Big|\mathbb E_u[f(X_{\tau_L})]-\mathbb E_v[f(X_{\tau_L})]\Big|\leq \eta \max\{f(w):w\in \Lambda_L\}.
	\end{equation}
\end{Prop}
\begin{proof}

Using \cite[Lemma~2.4.3]{LawlerLimicRandomWalks2010}, we find a coupling $(X^u,X^v)\sim\mathbf P$ and $C_1=C_1(d)>0$ such that the following holds: $X^u$ has law $\mathbb P_u$, $X^v$ has law $\mathbb P_v$, and for every $n\geq 1$,
\begin{equation}\label{eq:estimate srw 3}
	\mathbf P[X^u_k=X^v_k \: \text{for every }k\geq n]\geq 1-\frac{C_0\ell}{\sqrt{n}}.
\end{equation}
We denote by $\tau^u_L$ and $\tau^v_L$ the corresponding stopping times for the exit of $\Lambda_{L-1}$.
Now, observe that $\mathbf P$-almost surely, one has $\tau_L^{z}\geq L-\ell$ for $z\in \{u,v\}$. Hence, if $X^u_{L-\ell}=X^v_{L-\ell}$, then $X^u_{\tau_L^u}=X^v_{\tau_L^v}$. Combining this observation and \eqref{eq:estimate srw 3} yields
\begin{align}\notag
	\Big|\mathbb E_u[f(X_{\tau_L})]-\mathbb E_v[f(X_{\tau_L})]\Big|&=\Big|\mathbf E[f(X^u_{\tau_L^u})-f(X^v_{\tau_L^v})]\Big|
	\\&\leq 2\mathbf P[\{X_{k}^u=X^v_{k}\: \text{for every }k\geq L-\ell\}^c]\max\{f(w):w\in \Lambda_L\}\notag
	\\&\leq \frac{2C_0\ell}{\sqrt{L-\ell}}\max\{f(w):w\in \Lambda_L\}.
\end{align}
The proof follows from choosing $L_0$ large enough so that $\frac{2C_0\ell}{\sqrt{L_0-\ell}}\leq \eta$.
\end{proof}

\bibliographystyle{alpha}
\bibliography{biblio.bib}

\begin{thebibliography}{BDCGS12}

\bibitem[BBS14]{BauerschmidtBrydgesSlade2014Phi4fourdim}
Roland Bauerschmidt, David~C. Brydges, and Gordon Slade.
\newblock Scaling limits and critical behaviour of the $4$-dimensional
  $n$-component $|\varphi|^4$ spin model.
\newblock {\em Journal of Statistical Physics}, \textbf{157}:692--742, 2014.

\bibitem[BBS15a]{BauerschmidtBrydgesSlade2015WSAW4D}
Roland Bauerschmidt, David~C. Brydges, and Gordon Slade.
\newblock Critical two-point function of the 4-dimensional weakly self-avoiding
  walk.
\newblock {\em Communications in Mathematical Physics}, \textbf{338}:169--193,
  2015.

\bibitem[BBS15b]{BauerschmidtBrydgesSlade2015WSAW4DLogCorrections}
Roland Bauerschmidt, David~C. Brydges, and Gordon Slade.
\newblock Logarithmic correction for the susceptibility of the 4-dimensional
  weakly self-avoiding walk: a renormalisation group analysis.
\newblock {\em Communications in Mathematical Physics}, \textbf{337}:817--877,
  2015.

\bibitem[BBS19]{BauerschmidtBrydgesSladeBOOKRG2019}
Roland Bauerschmidt, David~C. Brydges, and Gordon Slade.
\newblock {\em Introduction to a {R}enormalisation {G}roup {M}ethod}, volume
  \textbf{2242}.
\newblock Springer Nature, 2019.

\bibitem[BDCGS12]{BauerschmidtDCGoodmanSlade}
Roland Bauerschmidt, Hugo Duminil-Copin, Jesse Goodman, and Gordon Slade.
\newblock Lectures on self-avoiding walks.
\newblock {\em Probability and Statistical Physics in Two and More Dimensions
  (D. Ellwood, CM Newman, V. Sidoravicius, and W. Werner, eds.), Clay
  Mathematics Institute Proceedings}, \textbf{15}:395--476, 2012.

\bibitem[BFF84]{BovierFelderFrohlich1984BubbleFiniteWSAW}
Anton Bovier, Giovanni Felder, and J{\"u}rg Fr{\"o}hlich.
\newblock On the critical properties of the {E}dwards and the self-avoiding
  walk model of polymer chains.
\newblock {\em Nuclear Physics B}, \textbf{230}(1):119--147, 1984.

\bibitem[BHH21]{BrydgesHelmuthHolmesContinuousLaceExpansion}
David Brydges, Tyler Helmuth, and Mark Holmes.
\newblock The continuous-time lace expansion.
\newblock {\em Communications on Pure and Applied Mathematics},
  \textbf{74}(11):2251--2309, 2021.

\bibitem[BHK18]{BolthausenvanderHofstadKozmaDummies2018}
Erwin Bolthausen, Remco W. van~der Hofstad, and Gady Kozma.
\newblock Lace expansion for dummies.
\newblock {\em Annales de l'Institut Henri Poincaré - Probabilités et
  Statistiques}, \textbf{54}(1):141--153, 2018.

\bibitem[BS85]{BrydgesSpencerSAW}
David Brydges and Thomas Spencer.
\newblock Self-avoiding walk in 5 or more dimensions.
\newblock {\em Communications in Mathematical Physics},
  \textbf{97}(1):125--148, 1985.

\bibitem[DCP]{DumPan24Perco}
Hugo Duminil-Copin and Romain Panis.
\newblock An alternative approach for the mean-field behaviour of spread-out
  {B}ernoulli percolation in dimensions $d>6$.
\newblock Preprint, \url{https://arxiv.org/abs/2410.03647}, 2024.

\bibitem[DCP25]{DuminilPanis2024newLB}
Hugo Duminil-Copin and Romain Panis.
\newblock New lower bounds for the (near) critical {I}sing and $\varphi^4$
  models’ two-point functions.
\newblock {\em Communications in Mathematical Physics}, \textbf{406}(3), 2025.

\bibitem[DCT16]{DuminilTassionNewProofSharpness2016}
Hugo Duminil-Copin and Vincent Tassion.
\newblock A new proof of the sharpness of the phase transition for {B}ernoulli
  percolation and the {I}sing model.
\newblock {\em Communications in Mathematical Physics}, \textbf{343}:725--745,
  2016.

\bibitem[FvdH17]{FitznervdHofstad2017Perco-d>10}
Robert Fitzner and Remco~W. van~der Hofstad.
\newblock Mean-field behavior for nearest-neighbor percolation in $d>10$.
\newblock {\em Electronic Journal of Probability}, \textbf{22}:43, 2017.

\bibitem[Har08]{HaraDecayOfCorrelationsInVariousModels2008}
Takashi Hara.
\newblock Decay of correlations in nearest-neighbor self-avoiding walk,
  percolation, lattice trees and animals.
\newblock {\em The Annals of Probability}, \textbf{36}(2):530--593, 2008.

\bibitem[HS90a]{HaraSlade1990Perco}
Takashi Hara and Gordon Slade.
\newblock Mean-field critical behaviour for percolation in high dimensions.
\newblock {\em Communications in Mathematical Physics},
  \textbf{128}(2):333--391, 1990.

\bibitem[HS90b]{HaraSlade1990LatticeTrees}
Takashi Hara and Gordon Slade.
\newblock On the upper critical dimension of lattice trees and lattice animals.
\newblock {\em Journal of statistical physics}, \textbf{59}:1469--1510, 1990.

\bibitem[HS92]{HaraSlade1992SAW}
Takashi Hara and Gordon Slade.
\newblock Self-avoiding walk in five or more dimensions {I}. {T}he critical
  behaviour.
\newblock {\em Communications in Mathematical Physics},
  \textbf{147}(1):101--136, 1992.

\bibitem[HvdHS03]{HaraSladevdHofstad2003PercoSO}
Takashi Hara, Remco~W. van~der Hofstad, and Gordon Slade.
\newblock Critical two-point functions and the lace expansion for spread-out
  high-dimensional percolation and related models.
\newblock {\em The Annals of Probability}, \textbf{31}(1):349--408, 2003.

\bibitem[Liu25]{Liu2023Plateau}
Yucheng Liu.
\newblock A general approach to massive upper bound for two-point function with
  application to self-avoiding walk torus plateau.
\newblock {\em Electronic Journal of Probability}, \textbf{30}:1--23, 2025.

\bibitem[LL10]{LawlerLimicRandomWalks2010}
Gregory~F. Lawler and Vlada Limic.
\newblock {\em Random walk: a modern introduction}, volume \textbf{123}.
\newblock Cambridge University Press, 2010.

\bibitem[MS93]{MadrasSlade2013SAW}
Neal Madras and Gordon Slade.
\newblock {\em The Self-Avoiding Walk}.
\newblock Birkh{\"a}user, Boston, (1993).

\bibitem[Pan]{PanisTriviality2023}
Romain Panis.
\newblock Triviality of the scaling limits of critical {I}sing and $\varphi^4$
  models with effective dimension at least four.
\newblock Preprint, \url{https://arxiv.org/abs/2309.05797}, 2023.

\bibitem[Sak07]{Sakai2007LaceExpIsing}
Akira Sakai.
\newblock Lace expansion for the {I}sing model.
\newblock {\em Communications in Mathematical Physics},
  \textbf{272}(2):283--344, 2007.

\bibitem[Sak15]{Sakai2015Phi4}
Akira Sakai.
\newblock Application of the lace expansion to the $\varphi^4$ model.
\newblock {\em Communications in Mathematical Physics}, \textbf{336}:619--648,
  2015.

\bibitem[Sak22]{Sakai2022correctboundsIsing}
Akira Sakai.
\newblock Correct bounds on the {I}sing lace-expansion coefficients.
\newblock {\em Communications in Mathematical Physics},
  \textbf{392}(3):783--823, 2022.

\bibitem[Sim80]{SimonInequalityIsing1980}
Barry Simon.
\newblock Correlation inequalities and the decay of correlations in
  ferromagnets.
\newblock {\em Communications in Mathematical Physics},
  \textbf{77}(2):111--126, 1980.

\bibitem[Sla87]{Slade1987Diffusion}
Gordon Slade.
\newblock The diffusion of self-avoiding random walk in high dimensions.
\newblock {\em Communications in Mathematical Physics}, \textbf{110}:661--683,
  1987.

\bibitem[Sla06]{SladeSaintFlourLaceExpansion2006}
Gordon Slade.
\newblock {\em The {L}ace {E}xpansion and {I}ts {A}pplications: {E}cole
  D'{E}t{\'e} de {P}robabilit{\'e}s de {S}aint-{F}lour {XXXIV}-2004}.
\newblock Springer, 2006.

\bibitem[Sla22]{Slade2022NewProofCVLACEExpansion}
Gordon Slade.
\newblock A simple convergence proof for the lace expansion.
\newblock In {\em Annales de l'Institut Henri Poincaré (B) Probabilites et
  Statistiques}, volume \textbf{58}, pages 26--33, 2022.

\bibitem[Sla23]{Slade2023near}
Gordon Slade.
\newblock The near-critical two-point function and the torus plateau for weakly
  self-avoiding walk in high dimensions.
\newblock {\em Mathematical Physics, Analysis and Geometry}, \textbf{26}(1):6,
  2023.

\end{thebibliography}

\end{document}